%% file: disc.tex
\documentclass[12pt]{amsart}
\pdfoutput=1
\usepackage
	[asymmetric,a4paper]
	{geometry}  
\usepackage{amssymb}
\usepackage{hyperref}
\input{diagxy}

\newtheorem{theorem}{Theorem}[section]
\newtheorem{lemma}[theorem]{Lemma}
\newtheorem{corollary}[theorem]{Corollary}
\newtheorem{proposition}[theorem]{Proposition}

\theoremstyle{definition}
\newtheorem{definition}[theorem]{Definition}

\newtheorem{notation}[theorem]{Notation}
\newtheorem{observation}[theorem]{Observation}

\newcommand{\op}{{\operatorname{op}}}
\newcommand{\Obj}{\operatorname{Ob}} 
 
\newcommand{\Id}{\mathrm{Id}}
\newcommand{\SuchThat}{\mathrm{\;\colon\:}}

\newcommand{\Slot}{\mathbin{\,\underline{\phantom{x}}\,}} 
\newcommand{\FiberOver}{\ast}
\newcommand{\CatName}[1]{#1}
\newcommand{\Forest}{\CatName{Forest}}
\newcommand{\Tree}{\CatName{Tree}}
\newcommand{\GlobDomain}{\mathbf{G}}
\newcommand{\GlobCard}{\CatName{GlobCard}}

\newcommand{\OGraph}{\CatName{OGraph}}
\newcommand{\Graph}{\CatName{Graph}}
\newcommand{\Pidi}{{t\Delta_+}}
\newcommand{\Disc}{{t\mathcal{I}_+}}
\newcommand{\iPidi}{{i\Delta_+}}
\newcommand{\iDisc}{{i\mathcal{I}_+}} 
\newcommand{\PidiNT}{{t\Delta}}
\newcommand{\DiscNT}{{t\mathcal{I}}}
\newcommand{\iPidiNT}{{i\Delta}}
\newcommand{\iDiscNT}{{i\mathcal{I}}}
\newcommand{\DiscToiDiscFunctor}{\Xi_{\mathcal{I}}}
\newcommand{\PidiToiPidiFunctor}{\Xi_{\Delta}}
\newcommand{\FunctorGlobCardToOGraph}{\Gamma}
\newcommand{\FunctorOGraphToGlobCard}{\Gamma'}

\newcommand{\ObjMapiPidiToOGraph}{\Upsilon}

\newcommand{\GCrel}{\blacktriangleleft}
 
\newcommand{\ExceptFirst}[1]
	{#1\thinspace\backslash\mathrm{f}}
\newcommand{\Pred}[1]{\mathrm{p} #1}
\newcommand{\Suspend}{\operatorname{su}}
\newcommand{\Restrict}{\CatName{Con}\thinspace}
\newcommand{\TwoCommaCat}
	{\CatName{Cat}/_{\CatName{Set}}}
\newcommand{\CompThing}[2]{#1[#2]}
\newcommand{\DeComp}[2]{\circ_{#2} \CompThing{#1}{\Pred{#2},#2}}
\newcommand{\nComp}[2]
	{\thinspace\stackrel{#1}{\circ}_{#2}\thinspace}
\newcommand{\ZeroComp}[1]
	{\circ_{#1}^{\phantom{#1}}}
\newcommand{\SubThing}[2]{#1(#2)}
\newcommand{\SubThingText}[2]
	{{restriction} of $#1$ by $#2$}
\newcommand{\SuspendThing}[2]{\Suspend(#1,#2)}
\newcommand{\SuspendThingText}[2]
	{{suspension} of $#1$ over $#2$}

\newcommand{\Disks}{\CatName{Disk}}
\newcommand{\DisksA}{\CatName{Disk}_+}

\newcommand{\Textendpoint}{end point}

\newcommand{\Textendelement}{end element}

\newcommand{\FreeFunctor}{\mathfrak{F}}
\newcommand{\ForgetfulFunctor}{\mathfrak{U}}
\newcommand{\FreeIsom}{{L}\thinspace}
\newcommand{\nCat}[1]{#1\text{-}\CatName{Cat}}
\newcommand{\wCat}{\omega\CatName{Cat}}
\newcommand{\TextMinusOne}{~$\text{-1}$}
\newcommand{\TextZero}{~$\text{0}$ (zero)}

\newcommand{\TextCategory}[1]
	{\MathText{#1}{-category}}
\newcommand{\TextCategories}[1]
	{\MathText{#1}{-categories}}
\newcommand{\TextFunctor}[1]
	{\MathText{#1}{-functor}}
\newcommand{\TextFunctors}[1]
	{\MathText{#1}{-functors}}

\newcommand{\MathText}[2]{$#1$#2}

\newcommand{\GCsrc}[1]{\mathrm{s}_{#1}\thinspace}
\newcommand{\GCtgt}[1]{\mathrm{t}_{#1}\thinspace}
\newcommand{\GCdom}[1]{\mathrm{dom}_{#1}\thinspace}
\newcommand{\GCcod}[1]{\mathrm{cod}_{#1}\thinspace}
\newcommand{\Root}[1]{\Obj #1}
\newcommand{\X}{{\phantom{X}}}

\title{On the duality between trees and disks}
\author{David Oury}
\address{Macquarie University\\New South Wales 2109\\Australia}
\email{david.oury@mq.edu.au}
\date{\today}
\thanks{The author wishes to thank Macquarie University for financial support (MQRES scholarship and travel grants to Northwestern University May 2009 and University of Adelaide June 2009).}

\begin{document}
\maketitle
\tableofcontents


\begin{abstract}
A combinatorial category $\Disks$ was introduced by Andr\'{e} Joyal to play a role in his definition of weak \TextCategory{\omega}. 
He defined the category ~$\Theta$ to be dual to $Disks$. In the ensuing literature, a more concrete description of ~$\Theta$ was provided. 
In this paper we provide another proof of the dual equivalence and introduce various 
categories equivalent to $\Disks$ or ~$\Theta$, each providing a helpful viewpoint. 
\end{abstract}

	\input{introduction}

	\input{ordint}

\input{definitions}

	\input{theorem121}
\input{globcard}
	\input{omegacats}
	\input{theorem222}

	\input{general}

	\input{bibliography}

\end{document}

%% file: introduction.tex
	\section{Introduction}
	
Andr\'{e} Joyal in order to define weak ~$n\text{-categories}$ introduced ~$\Theta$ which was naturally filtered with the simplicial category ~$\Delta$ being the first term of the filtration. %
In \cite{AJ_DDTC} he defined ~$\Theta$ as the dual of a category $\Disks$ of disks. %
He also suggested a more explicit description of ~$\Theta$ involving the trees of Michael Batanin in \cite{MB_UPMT}. %
Michael Makkai and Marek Zawadowski in \cite{MZ_DSCD} and Clemens Berger in \cite{CB_CNHC} gave explicit proofs that the two version of ~$\Theta$ are equivalent. %
In this paper we give a third proof which is a conceptual lifting of the duality between ordinals and intervals. %
In the process, several categories are introduced, each turning out to be equivalent to ~$\Theta$ or ~$\Disks$, and so each providing us with useful new perspectives on Joyal's definitions. %

The category $\Delta_+$ is known as the augmented 
simplicial category and contains a single object (the 
empty ordinal) in addition to those of $\Delta$. 
Primarily, we work with \emph{augmented}
categories which contain a unique \emph{trivial} 
object and are more suitable for inductive arguments. %
Four of these augmented categories have reduced
counterparts which are
equivalent to the categories ~$\Disks$ or ~$\Theta$. %

In Section \ref{sec: ord int} we recall the definitions of ordinals and intervals and define functors which witness that they are dual. %
The section ends with two simple results which are used 
in the proof of Theorem \ref{thm: 222}. %
In Section \ref{sec: induction} we define augmented
categories ~$\iDisc$ and ~$\iPidi$, inductively built 
from ~$\mathcal{I}_+$ and ~$\Delta_+$ (respectively), 
and prove that they are dual. %
Their reduced counterparts 
are denoted $\iDiscNT$ and $\iPidiNT$ (respectively). %
In Section \ref{sec: equivalences} we recall the definition of Joyal's category ~ $\Disks$ and demonstrate an equivalence between $\Disks$ and ~$\iDiscNT$. 

In Section \ref{sec: globular cardinals} we recall Street's
definition of globular cardinal and define restriction and suspension operations on them. %
In Section \ref{sec: ordinal graphs} we define so called ordinal graphs which are inductively defined counterparts to globular cardinals and demonstrate an equivalence between the category of globular cardinals and the category of ordinal graphs. %
In Section \ref{sec: omega cats} we recall the definition of \TextCategory{\omega} and define free functors on globular cardinals and on ordinal graphs. %
We then demonstrate that the free \TextCategory{\omega} on a globular cardinal is isomorphic to the free \TextCategory{\omega} on the corresponding ordinal graph. %
The objects of $\Theta$ are, by definition, the free $\omega\text{-categories}$ on globular cardinals. %
In Section \ref{section: iPidi and theta} we demonstrate
an equivalence between ~ $\Theta$ and the inductively defined category $\iPidiNT$. 

In the remaining two sections 
we provide a more categorical description 
of the categories $\Disks$ and $\Theta$ by defining so called labeled trees which satisfy specific requirements relevant to our purposes. %
Two categories of labeled trees, named $\Disc$ 
and $\Pidi$, are defined and easily shown to be dual. %
Their reduced counterparts are shown to be equivalent 
to the categories $\Disks$ and ~ $\Theta$ (respectively). %

Restriction and suspension operations are defined on labeled trees with the goal of working with them inductively and constructing equivalences between
these two categories of labeled trees and
their inductively defined counterparts. %
The proofs that ~$\iDisc$ is equivalent to ~$\Disc$ 
(Proposition \ref{prop: Disc to iDisc}) and 
that ~$\iPidi$ is equivalent to ~$\Pidi$ 
(Proposition \ref{prop: Pidi to iPidi}) are essentially 
the same; however, we give all the details for clarity. 

A disk of dimension ~$\le{N}$ is defined in \cite{AJ_DDTC} as a \emph{sequence of length ~$N$ of bundles of intervals} with extra conditions. %
If we had been dealing with families, rather than those bundles, of intervals we could have made use of known properties of the finite coproduct completion functor ~${Fam}_\Sigma$. %
In particular, if we have an equivalence 
~$\mathcal{A}^\op\simeq\mathcal{B}$, it lifts to an
equivalence ~$
	{Fam}_\Sigma(\mathcal{A})^\op
	\simeq
	{Fam}_\Pi(\mathcal{B})$ where ~${Fam}_\Pi
$ is the finite product completion functor. %
Our replacements ~$\DiscNT$ and ~$\PidiNT$ for ~$\Disks$ and ~$\Theta$ are modifications of ~${Fam}_\Sigma(\mathcal{I}_+)$ and ~${Fam}_\Pi(\Delta_+)$. %
Instead of finite families we have labeled trees. %



%% file: ordint.tex
	\section{The ordinal/interval duality}

	\label{sec: ord int}

Let $\CatName{Ord}$ be the sub-2-category of $\CatName{Cat}$ consisting of the ordered sets. %
A full subcategory ~$\Delta_+$ of ~$\CatName{Ord}$, 
called \emph{the algebraist's ~$\Delta$},  
has objects ~$[n]=\{0,..,n\}$ for ~$n$ 
in ~$\{-1, 0, ...\}$ 
where ~$[-1]=\{\}$. %
The category $\Delta_+$ is monoidal with tensor
given by ordinal addition, denoted $+$, and unit 
object given by the ordinal $[-1]$. %
The category ~$\mathcal{I}_+$ is the sub-category of ~$\Delta_+$ whose objects, called \emph{intervals}, are non-empty ordinals and whose morphisms preserve 
the greatest and least elements. %
Collectively, the greatest and least elements of $[n]$ 
are its \emph{\Textendpoint{s}}. %

A morphism in $\CatName{Ord}$ whose domain is a complete
lattice has a left adjoint if and only if it preserves
the least upper bound and has right adjoint if and only 
if it preserves the greatest lower bound. 
In particular, a morphism $
	\gamma\colon[m]\rightarrow[n]
$ has left adjoint $
	\gamma^\ell\colon[n]\rightarrow[m]
$ in $\CatName{Ord}$ if and only if it preserves the
largest element and has right adjoint $
	\gamma^r\colon[n]\rightarrow[m]
$ if and only if it preserves the smallest
element. 
The formula for $\gamma^\ell$ is \[
	\gamma^\ell(j) 
=	\min\{i\in[m]\SuchThat{j}\le\gamma(i)\}
\]
and for $\gamma^r$
is \[
	\gamma^r(j) 
=	\max\{i\in[m]\SuchThat\gamma(i)\le{j}\}. %
\]
Notice that $\gamma^\ell(j)=0$ if and only 
if $j=0$ and $\gamma^r(j)=m$ if and only 
if $j=n$. %
Given an interval map $
	f\colon[m]\rightarrow[n]
$ we have the commutative square
\[\bfig
	\square(0,0)|alrb|/->`->`->`{->}/<1000,400>%
	[\lbrack n-1 \rbrack
	`\lbrack m-1 \rbrack
	`\lbrack n \rbrack
	`\lbrack m \rbrack
	;f^\vee = f^r - \lbrack 0\rbrack
	`\partial_n
	`\partial_m
	`f^r
	]
\efig\]
defining $f^\vee$ in $\CatName{Ord}$. %
Given an ordinal map $
	g\colon[m]\rightarrow[n]
$ we have the commutative square
\[\bfig
	\square(0,0)|alrb|/{<-}`<-`<-`<-/<1000,400>%
	[\lbrack n+1 \rbrack
	`\lbrack m+1 \rbrack
	`\lbrack n \rbrack
	`\lbrack m \rbrack
	;g+\lbrack 0 \rbrack
	`\partial_{n+1}
	`\partial_{m+1}
	`g
	]
\efig\]
and define $g^\wedge$ as $
	(g+\lbrack 0 \rbrack)^\ell
$ in ${Ord}$. %

	\begin{theorem} \label{theorem: ord/int iso}
	
The functor \[
	(\Slot)^\vee
	\colon\mathcal{I}_+^\op 
	\rightarrow\Delta_+
\] which is defined by $
	[m]^\vee = [m-1]
$ and $
	\gamma^\vee
$ as above is an isomorphism of categories with
inverse ~$
	(\Slot)^\wedge
	\colon\Delta_+^\op 
	\rightarrow\mathcal{I}_+
$ defined by ~$
	[m]^\wedge = [m+1]
$ and $
	\gamma^\wedge
$ as above. %
\end{theorem}

	\begin{observation} \label{obs: the pin}
	
Let ~$\gamma\colon[m]\rightarrow[n]$ be 
an ordinal morphism. %
The fiber of ~$\gamma^\wedge$ over ~$j$ is
\begin{align*}
\{\gamma(j-1)+1, \gamma(j-1)+2, \ldots,\gamma (j) \}.
\end{align*}
This fact is used in Theorem \ref{thm: 222}. %
\end{observation}

	\begin{observation}

	\label{obs: outside sent to endpoints}	
	
For an ordinal map $\gamma\colon[m]\rightarrow[n]$ then
$\gamma^{\wedge}(i)$ is an endpoint when either $
	i\le\min\text{im}\gamma
$ or $
	i>\max\text{im}\gamma
$. %
This fact is also used in Theorem \ref{thm: 222}.

\end{observation}

%% file: definitions.tex
%
	\section{Induction on intervals and ordinals}
	
	\label{sec: induction}


In this section we introduce inductively defined augmented 
categories ~$\iDisc$ and ~$\iPidi$ and demonstrate 
that they are dual. %
	
	\begin{definition}

We define the category ~$\iDisc$ inductively. %
The object of \emph{height} 0 (zero) is the 
interval ~$[0]$ and is \emph{trivial}. %
An object ~$H$ of \emph{height} ~$n$ is an 
interval ~$\Root{H}$ and for each ~$i$ in $\Root{H}$ 
an object ~$H(i)$
of height strictly less than ~$n$ which is trivial if and only if ~$i$ is an endpoint of ~$\Root{H}$. %

For every object ~$H$ there is a unique morphism ~${H}\rightarrow[0]$. %
Hence~ ~$[0]$ is terminal. %
A morphism ~$
	{g}\colon{H}\rightarrow{K}$ 
consists of an interval map ~$
	{g}\colon{\Root{H}}\rightarrow{\Root{K}}
$ and for all ~$i$ in ~$\Root{H}$ a  morphism $
	{g(i)}\colon{H(i)}\rightarrow{K(g i)}
$. %

Define composition using induction as follows. %
Let ~ $f\colon{H}\rightarrow{K}$ and ~ $
	g\colon{K}\rightarrow{L}
$ be composable morphisms. %
The object map of ~ $g\circ{f}$ is the composite of the object maps of ~$g$ and ~$f$. %
For ~$i\in\Root{H}$ then the composite ~
$	(g\circ{f})(i)
	\colon H(i) \rightarrow L(gfi)
$ is ~$ g(fi) \circ f(i)$. %
We have the category ~$\iDisc$. %
The category ~$\iDiscNT$ is the full subcategory 
of ~$\iDisc$ containing the non-trivial objects. %
\end{definition}
	
	\begin{definition}
	
We define the category ~$\iPidi$ inductively. %
The object 
of \emph{height} 0 (zero) is the ordinal ~$[-1]$ and is \emph{trivial}. %
An object 
~$K$ of \emph{height} ~$n$ is an ordinal ~$\Root{K}$ and for each ~$
	j\in (\Root{K})^{\wedge}
$ an object ~$K(j)$
of height strictly less than ~$n$ which is trivial if and only if ~$j$ is an endpoint of ~$
	(\Root{K})^{\wedge}
$. %

For every object 
~$K$ there is a unique morphism ~$[-1]\rightarrow{K}$. %
Hence~ ~$[-1]$ is initial. %
A morphism 
$	{g}\colon{H}\rightarrow{K}$ 
consists of an ordinal map~
$	{g}\colon\Root{H}\rightarrow\Root{K}$ 
and for all ~$j\in (\Root{K})^{\wedge}$ 
a morphism 
~$	{g(j)}
	\colon{H(g^{\wedge}j)}
	\rightarrow{K(j)}
$. %

Define composition using induction as follows. %
Let ~ $f\colon{H}\rightarrow{K}$ and ~ $g\colon{K}\rightarrow{L}$ be composable morphisms. 
The object map of ~ $g\circ{f}$ is the composite of the object maps of ~$g$ and ~$f$. %
For ~$i\in\Root{L}$ then the composite ~
$	(g\circ{f})(i)
	\colon H((g\circ f)^\wedge i) 
	\rightarrow L(i)
$ is ~$ g(i) \circ f(g^\wedge i)$. %
We have the category ~$\iPidi$.
The category ~$\iPidiNT$ is the full subcategory 
of ~$\iPidi$ containing the non-trivial objects. %
\end{definition}

	\begin{definition}
	
We define functors 
\[	\vee\colon\iDisc^\op\rightarrow\iPidi
\]
and
\[	\wedge\colon\iPidi^\op\rightarrow\iDisc
\] 
using the functors $(\Slot)^{\vee}$ and $(\Slot)^{\wedge}$ of the equivalence between ordinals and intervals. %
	
Define $\vee$ on objects using induction on their height. %
Send the trivial object 
$[0]$ of $\iDisc$ to the trivial object 
$[-1]$ of $\iPidi$. %
Assume ~$\vee$ is defined for objects of height ~$n$ and let ~$H$ 
be an object of height~ $n+1$. %
Define ~ $\vee H$ as ~ 
$	((\Root{H})^{\vee}, \vee H(i) )
$. %

Define $\vee$ on morphisms using induction on the height 
of their codomain. Send each unique morphism of $\iDisc$
into the trivial object $[0]$ to the corresponding 
unique morphism of $\iPidi$ 
out of the trivial object ~$[-1]$. %
Assume ~$\vee$ is defined on morphisms 
with codomain
of height ~$n$ and let ~$g$ be a morphism of $\iDisc$
with codomain of height~ $n+1$. 
Define ~ $\vee g$ as ~ 
$	(g^{\vee}, \vee g(i) )
$ 
as each morphism ~$g(i)$ has codomain of height~ $n$. 

Similarly, define $\wedge$ on objects (respectively morphisms) using induction on their height (respectively on the height of their domains). %
Send the trivial object of $\iPidi$ 
to the trivial object of $\iDisc$. 
Send morphisms of $\iPidi$ 
out of the trivial object to morphisms of $\iDisc$
into the trivial object. %
Let~
$	g
$
be a morphism of $\iPidi$. 
Define ~ $\wedge g$ as ~ 
$	(g^{\wedge},\wedge g(i))
$. 

\end{definition}

	\begin{theorem} 
	
	\label{thm: 122 iDisc}
	
The functors $\vee$ and $\wedge$ are mutually inverse isomorphisms. 
\end{theorem}

The proof uses induction and follows directly from the mutually inverse functors of the duality between ordinals and intervals. 

	\begin{corollary}
The categories ~$\iDiscNT$ and ~$\iPidiNT$ dual. 	
\end{corollary}

%% file: theorem121.tex
	\section{Equivalence between $\iDiscNT$ and $\Disks$}
	
	\label{sec: equivalences}

In Section \ref{sec: induction} we showed  that the categories ~$\iDiscNT$ and ~$\iPidiNT$ are dual. %
Here we demonstrate that the category ~$\iDiscNT$ 
and Joyal's category $\Disks$ are equivalent. %
To do so we construct the category ~$\DisksA$, the
augmented counterpart to ~$\Disks$, and show that
it is equivalent to the category ~$\iDisc$. %
We begin by defining \emph{forests} and \emph{trees},
and then define operations of restriction and suspension
on trees in order to work with them inductively. %

	\begin{definition} 
	
	\label{def: forest}

A \emph{forest} is a functor	
~$	{A}\colon\omega^\op\rightarrow\CatName{Set}
$ %
\[
\cdots \to/->/<500>^{p_2} 
A_2 \to/->/<500>^{p_1} 
A_1 \to/->/<500>^{p_0}
A_0.
\]
A \emph{tree} ~$A$ is a forest such that 
~$	A_{0}\cong\{\ast\}
$. %
The vertices of \emph{height} ~$n$ are the elements of ~$A_n$. 
The unique vertex of height \TextZero\ of a tree is called the \emph{root}. 
We often denote the root of a tree ~ $A$ as ~ $\ast$. %
We can regard a tree as a directed graph. There is an edge from vertex ~$x$ to vertex ~$y$ when ~$p_n(x)=y$ for some~$n\in\mathbb{N}$. 
We denote the above forest by ~$(A, p)$ or $A$. %
Define ~$p_{n,m}\colon{A_{n+m}}\rightarrow{A_{n}}$ as the composite~$\circ_{i=n}^{n+m-1} p_i$ which is
\[	
	A_{n+m} \to/->/<500>^{p_{n+m-1}} 
	\cdots \to/->/<500>^{p_{n}}
	A_{n}. %
\]
A forest has \emph{degree} ~$n$ when ~$p_{n,m}$ is a bijection for all ~$m\ge{1}$ and has \emph{finite degree} when it has degree some $n\in\mathbb{N}$. %
\end{definition}

	\begin{definition} 
	
	\label{def: tree map}
	
A \emph{forest map} is a natural transformation of forests and so is a sequence of set maps 
$$\bfig
\square|ammb|/->``->`->/[\cdots`A_2`\cdots`B_2;p_2`\cdots`f_2`q_2]
\square(500,0)|ammb|/->`->`->`->/[A_2`A_1`B_2`B_1;p_1`f_2`f_1`q_1]
\square(1000,0)|ammb|/->`->`->`->/[A_1`A_0`B_1`B_0;p_0`f_1`f_0`q_0]
\efig$$
such that the squares commute. %
This map is denoted by 
~$	f\colon(A,p)\rightarrow(B,q)
$ or 
~$	f\colon{A}\rightarrow{B}
$. %
A \emph{tree map} is a natural transformation between trees. %
We have the category $\CatName{Forest}$ and its full subcategory ~$\Tree$. %
\end{definition}

	\begin{definition}
	
	\label{def: subtree}

We define a \emph{restriction} operation on trees
in order to work with trees as inductive or recursive objects. %
Let 
~$	{u_n}\colon\omega\rightarrow\omega
$
be the functor defined by 
~$	{u_n}(i)=i+n
$. %
Let ~$A$ be a forest and ~$x$ an element of ~$A_{n}$. %
The \emph{\SubThingText{A}{x}} denoted 
~$	\SubThing{A}{x}
$ 
is the largest subfunctor of 
~$	A\circ u_{n}
	\colon\omega^\op
	\rightarrow\CatName{Set}
$
such that 
~$	\SubThing{A}{x}_{0}=\{x\}
$. %
We sometimes refer to ~$\SubThing{A}{x}$ as a \emph{subtree} of ~$A$. %

The \emph{\SubThingText{f}{x}} denoted ~ 
$	\SubThing{f}{x}
$, where
~$	f\colon{A}\rightarrow{B}
$ 
is a forest map, 
is the lifting in 
\[\bfig
\square(0,0)|alrb|/-->`->`->`->/<700,300>%
	[\SubThing{A}{x}
	`\SubThing{B}{f_n x}
	`A\circ u_n`B\circ u_n
	;\SubThing{f}{x}
	`\mathrm{incl}
	`\mathrm{incl}
	`f\cdot u_n]
\efig
\] 
of ~$f\cdot u_n$ along the inclusions 
~$	\SubThing{A}{x} \rightarrow A\circ u_n
$ 
and 
~$	\SubThing{B}{f_n x} 
	\rightarrow B\circ u_n
$. %

\end{definition}

	\begin{definition}
	
	\label{def: suspension tree}

Define a \emph{suspension} functor 
\[	\operatorname{su}
	\colon\CatName{Forest}
	\rightarrow\Tree
\]
as follows. %
Given a forest ~$A$ then its \emph{suspension} ~$\operatorname{su}A$ is 
\[
\cdots \to/->/<500>^{p_1} 
A_1 \to/->/<500>^{p_0} 
A_0 \to/->/<500>
\{\ast\}.
\] %
Given a map 
~$	f\colon{A}\rightarrow{B}
$
of forests then its \emph{suspension} ~$\Suspend f$ is 
$$\bfig
\square(500,0)|amma|
	/{->}`{}`{->}`{->}/
	[\cdots`A_1`\cdots`B_1
	;\phantom{s}
		`\cdots
		`f_{1}
		`\phantom{s}
	]
\square(1000,0)|amma|
	/{->}`{}`{->}`{->}/
	[A_1`A_0`B_1`B_0
	;\phantom{s}
		`f_{1}
		`f_{0}
		`\phantom{s}
	]
\square(1500,0)|amma|
	/{->}`{}`{->}`{->}/
	[A_0
	`\phantom{.}\{\ast\}\phantom{.}
	`B_0
	`\phantom{.}\{\ast\}.
	;\phantom{s}
		`f_{0}
		`
		`\phantom{s}
	]
\efig$$

\end{definition}

%

	\begin{observation}
	
 	\label{obs: id = sub suspend}
	
The coproduct of a collection of trees is a forest
and its suspension is a tree. The subtrees of the 
suspension are isomorphic to the trees of the original collection. %
We provide the details below. %

Let $(A(i),p(i))$ be a tree with ~$A(i)_{0}=\{x_{i}\}$ for each ~$i$ in a set $I$. %
Let
~$	A'=\Suspend \sum A(i)
$ and let
~$	\mathrm{copr}(i)
	\colon A(i)
	\rightarrow\sum{A(i)}
$ 
be coprojections for each ~$i\in{I}$. %
The fiber of ~$p(i)_{0,n}$ over ~$x_i$ is ~$A(i)_n$. %
The fiber of ~$[\sum p(i)]_{0,n}$ over ~$x_i$ lying in ~$
	\sum A(i)_0$ is ~$\SubThing{A'}{x_i}_n
$. %
The coproduct in ~$\CatName{Set}$ requires that the 
former is sent by the left coprojection of
$$\bfig
\Square/->`->`->`-->/
	[A(i)_n`A(i)_0`\sum A(i)_n`\sum A(i)_0.
	;p(i)_0^n
	`\mathrm{copr}
	`\mathrm{copr}
	`\sum p(i)_{0,n}
	]
\efig$$
onto the latter. %
As the coprojections are monomorphisms then ~$A(i)$ and ~$\SubThing{A'}{x_i}$ are isomorphic by ~$\mathrm{copr}(i)$. %
\end{observation}

	\begin{definition}

	\label{def: disk}
	
A \emph{disk} as defined by Joyal in \cite{AJ_DDTC} is a tree $(A,p)$ of finite degree
\begin{enumerate}
\item \label{enum: fibers are intervals}
such that the fibers of $p_n\colon A_{n+1} \rightarrow A_{n}$ have interval structure (for $n\in\mathbb{N}$)
\item \label{enum: sections give endpoints} 
with sections 
$d_0,d_1\colon A_n\rightarrow A_{n+1}$ of $p_n$ 
where $p_n^{\FiberOver}(x)=[d_0(x),..,d_1(x)]$
\item \label{enum: equalizer}
such that the equalizer of $d_0,d_1\colon A_n\rightarrow A_{n+1}$ is $d_0(A_{n-1})\cup d_1(A_{n-1})$.
\end{enumerate}
The equalizer of condition \ref{enum: equalizer} is the \emph{singular set} of $A_n$. All fibers are non-empty by condition \ref{enum: sections give endpoints} and the interval $p_0^{\FiberOver}(x)$ is strict by condition \ref{enum: equalizer}, where ~$x$ is the single element of ~$A_0$. %

A \emph{morphism of disks} $f\colon(A,p)\rightarrow(B,q)$ as defined by Joyal in \cite{AJ_DDTC} as a sequence of set maps ~$f_n\colon A_n\rightarrow B_n$ which commute with the projections $p_n$ and $q_n$, respect the order of the interval fibers and preserve the endpoints (first and last elements) of the interval fibers. This defines the category ~$\Disks$. %

In addition, we define the augmented category $\DisksA$. %
An object of ~$\DisksA$ is a tree 
of finite 
degree satisfying \ref{enum: fibers are intervals} and \ref{enum: sections give endpoints} above, 
but we relax \ref{enum: equalizer} to allow the fiber 
over the root to be the unique non-strict 
interval ~$[0]$. %
Hence ~$\DisksA$ contains, in addition to the objects
of ~$\Disks$, trees of degree \TextZero. %
These additional objects are terminal. %
The morphisms of ~$\DisksA$ are defined identically 
to those of ~$\Disks$. %
\end{definition}

	\begin{definition} \label{def: disks}

We define a functor \[
	\Phi\colon\DisksA\rightarrow\iDisc.
\]	
Define ~$\Phi$ on objects using induction on the degree of disks. %
Send disks of degree \TextZero\ to ~$[0]$ the trivial object of $\iDisc$. %
Assume that ~$\Phi$ is defined on disks of degree ~$n$ 
and let ~$(A,p)$ be a disk of degree ~$n+1$. %
We define an object ~$H$ of ~$\iDisc$ from the 
data of ~$(A,p)$. %
Let ~$
	\Root{H}=p_{0}^{\FiberOver}(x)
$, the fiber over  the unique element ~$x$ of ~$A_{0}$, 
and let ~$
	H(i) = \Phi A(i)
$ for each ~$i\in\Root{H}$ where ~$\SubThing{A}{i}$ 
is the \SubThingText{A}{i}. %
Note that ~$\Root{H}$ is an interval as fibers have
interval structure, that ~$A(i)$ is a disk of 
degree ~$n$ for each ~$i\in \Root{H}$ and that ~$A(i)$ 
is trivial when ~$i$ is an endpoint by 
condition \ref{enum: equalizer}. %
Define ~ $\Phi A$ as ~ $H$. %

We define ~$\Phi$ on disk morphisms using induction on the degree of their codomain. %
The disks of degree \TextZero\ are terminal objects. %
As ~$\Phi$ preserves the terminal object then disk morphisms with codomain of degree \TextZero\ are sent to the unique morphism into the trivial object of $\iDisc$. 
Assume ~$\Phi$ is defined on morphisms with 
codomain of degree ~$n$ and let ~ 
$	{f}\colon(A,p)\rightarrow(B,q)
$ be a disk morphism with codomain
of degree ~$n+1$. %
We define a morphism $g$ of $\iDisc$ from the 
data of ~ $f$. %
Let ~ 
$	g=f_1
$ 
which is an interval morphism as ~ $f$ preserves order and endpoints. %
Let
~$	g(i) = \Phi f(i)
$
where ~$f(i)$ is the \SubThingText{f}{i} for each 
~$i\in A_1$. %
Define ~$\Phi{f}$ as $g$. %
\end{definition}

	\begin{theorem} 

	\label{thm: 121 iDisc}
		
The category $\DisksA$ is equivalent to the category $\iDisc$ by
\[	\Phi\colon\DisksA\rightarrow\iDisc
\]
which is surjective on objects.
\end{theorem}

\begin{proof}

	\textbf{Surjective.} 
We show that ~$\Phi$ is surjective on objects using induction on the height of objects of $\iDisc$. 
The trivial disks are sent to the trivial object ~
$[0]$ and so ~$\Phi$ is surjective on objects of height \TextZero. %
Assume ~$\Phi$ is surjective on objects
of height ~$n$ and let ~$H$ be an object
of height ~$n+1$. %
By induction there exists a disk ~$A(i)$ with ~$\Phi A(i)=H(i)$ for each ~${i}$ in ~ $\Root{H}$. %
Let 
~$	A'
=	\Suspend \sum A(i)
$ where the coproduct is indexed over 
the elements of ~ $\Root{H}$. %
Give the fiber over ~ $\ast$ an interval structure
from that of ~$\Root{H}$. %
For each element ~$x$ in ~$A'_n$ with ~$n\in\mathbb{N}_+$ give the fiber over ~$x$ an interval structure by pulling back along the coprojections. %
By Observation ~\ref{obs: id = sub suspend} then ~$A'(i)\cong A(i)$. %
Define ~$\Phi A'$ as ~ $H'$. %
As ~$\Phi$ is constant on isomorphism classes then 
we have ~$
	\Phi A'(i)=\Phi A(i)
$, equivalently ~$H'(i)=H(i)$. %
By construction ~$
	\Root{H'}=(p'_0)^\FiberOver(\ast)
$ is isomorphic to ~$\Root{H}$ and is in fact identically ~$\Root{H}$ as ~$\Delta$ is skeletal. %
Hence ~$\Phi$ is surjective on objects.

\textbf{Faithful.} We show that ~$\Phi$ is faithful using induction on the height of the codomain of 
morphisms of $\iDisc$. %
As ~$\Phi$ reflects the terminal object it is 
faithful on morphisms with codomain of 
height \TextZero. %
Assume that ~$\Phi$ is faithful on morphisms with
codomain of height ~$n$ and let ~$
	{f},{f'}\colon(A,p)\rightarrow(B,q)
$
be parallel disk morphisms such that ~$
	\Phi f = \Phi f'
$
has codomain of height ~$n+1$. %
Then, by induction, ~$
	f(i) = f'(i)
$ for each ~$
	i\in A_1
$. %
As ~$\SubThing{f}{i}$ and ~$\SubThing{f'}{i}$ are defined by liftings on inclusions which are jointly surjective it follows that ~$f=f'$ and ~$\Phi$ is faithful. %

\textbf{Full.} We show ~$\Phi$ is full using induction on the height of the codomain of morphisms of $\iDisc$. %
As ~ $\Phi$ preserves the terminal object it is
full on morphisms of height \TextZero. %
Assume that ~$\Phi$ is full on 
morphisms with codomain of height ~$n$ and let ~$
	{g}\colon\Phi(A,p)\rightarrow\Phi(B,q)
$ have codomain of height ~$n+1$. %
We have an interval morphism  ~$g$ and a 
morphism ~$g(i)$ of $\iDisc$ for each ~$i\in\Phi(A,p)$. %
By induction there exist disk morphisms ~$
	f(i)\colon{A(i)}\rightarrow{B(g i)}
$ such that ~$\Phi f(i) = g(i)$. %
In the following diagram
$$\bfig
\node m1(0,0)[A_{n+1}]
\node m2(700,0)[\sum \SubThing{A}{i}_n]
\node m3(1700,0)[\sum \SubThing{B}{g i}_n]
\node m4(2400,0)[B_{n+1}]
\node u2(700,500)[\SubThing{A}{i}_n]
\node u3(1700,500)[\SubThing{B}{g i}_n]
\arrow/->/[u2`m1;\mathrm{incl}]
\arrow|m|/->/[u2`m2;\mathrm{copr}]
\arrow/->/[u2`u3;f(i)_n]
\arrow|m|/->/[u3`m3;\mathrm{copr}]
\arrow/->/[u3`m4;\mathrm{incl}]
\arrow/-->/[m1`m2;]
\arrow/-->/[m2`m3;]
\arrow/-->/[m3`m4;]
\efig$$
the left triangle commutes as the inclusions into ~ $A_{n+1}$ are jointly surjective and the other regions commute by coproduct. %
Define ~$f'$ by setting ~$f'_{n+1}$ to the lower composite for each ~$n\in\mathbb{N}$. %
Then ~$\SubThing{f'}{i}$ is (by definition) the upper horizontal morphism and so 
~$	f(i)=\SubThing{f'}{i}$. %
The inclusions and commutativity of the diagram imply that ~$f_1$ as defined is identically ~$g$ as required. %
Hence ~$\Phi f'=g$. %
\end{proof}

	\begin{corollary}
	
The category $\Disks$ is equivalent to the 
category $\iDiscNT$. %
\end{corollary}

%% file: globcard.tex
	\section{Objects of $\Theta$}

From \cite{MB_UPMT}, an object of ~ $\Theta$ is defined as the free \TextCategory{\omega} on a non-empty globular cardinal. %
We define the augmented counterpart 
to ~$\Theta$, named ~$\Theta_+$, as containing objects 
which are the free \TextCategory{\omega} on a 
(possibly empty) globular cardinal. %

In Section \ref{sec: globular cardinals} we recall the
definition of globular cardinal and define 
restriction and suspension operations on them. %
In Section \ref{sec: ordinal graphs} we define 
ordinal graphs, which are inductively defined 
counterparts of globular cardinals. %
An equivalence is demonstrated between the categories of
globular cardinals and ordinal graphs. %
The section closes with the definition of a map that 
sends objects of $\iPidi$ to ordinal graphs. 
In Section \ref{sec: omega cats} we recall the 
definition of \TextCategory{\omega} and adapt Michael 
Batanin's construction of the free \TextCategory{\omega} 
on a globular set. %
In addition we define the free \TextCategory{\omega} 
on an ordinal graph and show that it is isomorphic 
to the free \TextCategory{\omega} of the corresponding
globular cardinal.

In this and following sections we will use the following shorthand with the intention of providing an uncluttered presentation. %
In the sequel often indices are given over all elements of a known finite linearly ordered set except the first. %
In some cases the previous element to the index is also used. 
Our shorthand is designed to simplify the presentation in these instances. %

	\begin{notation}

Let ~ $\CatName{FinOrd}$ denote the full subcategory of ~ $\CatName{Ord}$ with objects finite linearly ordered sets. %
Define 
\begin{align*}
\ExceptFirst{(\Slot)}
	& \colon	\Obj\CatName{FinOrd}
			\rightarrow\Obj\CatName{FinOrd}
\\	& \colon	\{x_0,x_1,\ldots,x_p\}
			\mapsto\{x_1, x_2, \ldots,x_p\}.
\end{align*}
which returns its argument without the first element. The ``$\backslash\mathrm{f}$'' is intended to indicate 
that the first element is removed. %
The operator $\Pred{(\Slot)}$ takes arguments which are elements of finite linearly ordered sets and returns their predecessor.  %
For example given a finite linearly ordered set
~ $I=\{x_0,x_1,\ldots,x_p\}$ then ~ 
$\Pred(x_i)=x_{i-1}
$ for all 
$	x_i\in\ExceptFirst{I}
$. %
\end{notation}

	\subsection{Globular cardinals}
	
	\label{sec: globular cardinals}

Globular cardinals were defined by Ross Street in \cite{RS_PTGS} Section 1. %
We collect here the definitions and notation related to this concept, as relevant for our purposes. %
In addition, we define restriction and suspension 
operations. 

	\begin{definition} \label{globular sets}

We begin by quoting the definition of globular object from \cite{RS_PTGS}. %
Let ~$\GlobDomain$ be the category with objects the natural numbers and non-identity arrows
\[	\sigma_{m},\tau_{m}\colon m\rightarrow n
	\quad\text{for }m<n
\]
such that 
\[\bfig	\Vtriangle[k`m`n;\beta_{k}`\beta_{k}`\alpha_{k}]
\efig\]
commutes for all ~$k<m<n$ and all ~$\alpha,\beta\in\{\sigma,\tau\}$. %
A \emph{globular object} in ~$\mathcal{C}$ is a functor ~$X\colon\GlobDomain^\op\rightarrow\mathcal{C}$. %
A \emph{morphism of globular objects} is a natural transformation between globular objects. %
Hence a \emph{globular set} is a pair of sequences of set maps
$$ \cdots 
\two/->`->/<500>^{s_2}_{t_2} X_2 \two/->`->/<500>^{s_1}_{t_1} X_1 \two/->`->/<500>^{s_0}_{t_0}  X_0$$
such that ~$s_n s_{n+1} = s_n t_{n+1}$ and ~$t_n s_{n+1} = t_n t_{n+1}$ for all ~$n\in\mathbb{N}$. %
The maps of the sequence ~$s$ are \emph{source maps} and the maps of the sequence ~$t$ are \emph{target maps}. %
The set ~$X_n$ contains the \emph{$n\text{-vertices}$} of ~$X$. %

A globular set $X$ has \emph{dimension} $n$ 
when ~$X_m$ is empty for all ~$m>n$. %
The empty globular set has dimension \TextMinusOne\ 
(minus one).
A globular set has \emph{finite dimension}
when it has dimension $n$
for some ~$n\in\mathbb{N}\cup\{-1\}$. %

A \emph{morphism ~$f\colon X\rightarrow Y$ of globular sets} is a sequence ~$f$ of set maps
$$\bfig
\square(500,0)|amma|
		/{@{>}@<3pt>}`{}`{->}`{@{>}@<3pt>}/
		[\cdots`X_2`\cdots`Y_2;s`\cdots`f_2`s]
\square(1000,0)|amma|
		/{@{>}@<3pt>}`{}`{->}`{@{>}@<3pt>}/
		[X_2`X_1`Y_2`Y_1;s`f_2`f_1`s]
\square(1500,0)|amma|
		/{@{>}@<3pt>}`{}`{->}`{@{>}@<3pt>}/
		[X_1`X_0`Y_1`Y_0;s`f_1`f_0`s]

\square(500,0)|bmmb|
		/{@{>}@<-3pt>}`{}`{->}`{@{>}@<-3pt>}/
		[\cdots`X_2`\cdots`Y_2;t`\cdots`f_2`t]
\square(1000,0)|bmmb|
		/{@{>}@<-3pt>}`{}`{->}`{@{>}@<-3pt>}/
		[X_2`X_1`Y_2`Y_1;t`f_2`f_1`t]
\square(1500,0)|bmmb|
		/{@{>}@<-3pt>}`{}`{->}`{@{>}@<-3pt>}/
		[X_1`X_0`Y_1`Y_0;t`f_1`f_0`t]

\efig$$
which commute with the source and target maps. %
We have the category 
of globular sets. %

A globular set ~$X$ has a partial order ~$\blacktriangleleft$ generated from the relation
\begin{eqnarray*}
	x \prec y 
	& \quad\text{when}\quad & 	x=s(y)
\\	& \quad\text{or}\quad &		y=t(x)
\end{eqnarray*}
for ~$x\in X_n$ and either 
~$	y\in X_{n+1}$ or ~$y\in X_{n-1}
$ 
(respectively). %
When ~$x\blacktriangleleft y$ there exists a (possibly trivial) sequence ~$x_0, ..., x_n$ of vertices in ~$X$ with 
~$	x_0=x
$, 
~$x_n=y
$ 
and 
~$x_{i-1}\prec x_i
$ for all 
~$	i\in\{1,..,n\}
$. %
A \emph{globular cardinal} is a globular set with a finite set of vertices and where the order given above is linear. We have ~$\GlobCard$ the category of globular cardinals. %
\end{definition}

	\begin{definition}

Vertices ~$x$ and ~$y$ in a globular cardinal ~$X$ are \emph{consecutive in ~$Y$} (with ~$Y$ a
subset of $X$) when 
~$	\{ z\in Y \SuchThat x \GCrel z 
						\GCrel y \}
=	\{x,y\}
$. %
	
\end{definition}

	\begin{observation}

	\label{obs: globcard incremental}

Let ~$f\colon{X}\rightarrow{Y}$ be a morphism of globular cardinals and let ~$x$ and ~$y$ be consecutive in ~$X_n$. %
There exists an element $z$ either in $X_{n+1}$
or ~$X_{n-1}$ such that ~
$s(z)=x$ and ~$t(z)=y$
or such that $s(y)=z$ and ~$t(x)=z$. 
In either case ~$f_n x$ and ~$f_n y$ are consecutive in ~ $Y_n$ as ~$f$ preserves source 
and target. %
Hence ~$f_n$ is injective and the image of ~$f_n$ is an interval. %
\end{observation}

	\begin{definition}

	\label{def: incremental GC}

A map 
~$	g\colon{A}
	\rightarrow{B}
$
of finite linearly ordered sets that is injective and whose image is an interval is called \emph{incremental}.
\end{definition}

	\begin{definition}

	\label{def: restriction GC}

We define a \emph{restriction} operation on globular cardinals. %
Given two consecutive ~\MathText{n}{-vertices} of a globular cardinal ~$X$ the operation returns the globular cardinal which we might call the hom-set determined by the two ~\MathText{n}{-vertices}. %
Let ~
$	{u_{n}}\colon\GlobDomain\rightarrow\GlobDomain
$
be the functor defined by ~${u_{n}}(i)=i+n$. %
Let ~$X$ be a globular cardinal with ~$y$ and ~$z$ consecutive vertices of ~$X_{n}$. %
Define the 
\emph{\SubThingText{X}{y,z}}
denoted ~$\SubThing{X}{y,z}$ as the largest subfunctor of 
~$	X\circ u_{n+1}
$ 
such that ~
$	y\GCrel x\GCrel z
$ for all ~ $x\in\SubThing{X}{y,z}$. %
Notice that ~$	X(\sigma)x=y
$ 
and 
~$	X(\tau)x=z
$ 
for all 
~$	x\in\SubThing{X}{y,z}_{0}\subseteq{X_{n+1}}
$. %

We have an inclusion
~$	\iota_{y,z}
	\colon\SubThing{X}{y,z}
	\rightarrow{X}\circ u_{n+1}^\X
$ 
of functors for every pair of consecutive vertices ~$y,z\in X_n$. %
The 
\emph{\SubThingText{f\colon{X}\rightarrow{Y}}{y,z}}
denoted
$	\SubThing{f}{y,z}	
$
is the lifting in 
\[\bfig
\square(0,0)|ammb|/-->`->`->`->/
	<900,300>[\SubThing{X}{y,z}
			`\SubThing{Y}{f_n^\X y,f_n^\X z}
			`X\circ u_{n+1}^\X`Y\circ u_{n+1}^\X
		;\SubThing{f}{y,z}```f\cdot u_{n+1}^\X]
\efig
\] 
of ~$f\cdot u_{n+1}^\X$ along the inclusions 
~$	\iota_{y,z}^\X
$ 
and 
~$	\iota_{f_{n}y,f_{n}z}^\X
$. %
\end{definition}

	\begin{definition}

	\label{def: suspension GC}

We define a \emph{suspension} operation on collections of globular cardinals. 
Let ${A}=\{x_0,x_1,\ldots,x_p\}$ 
be a finite linearly ordered set
and let $X(i)$ be a globular cardinal for each ~
$	i\in
	\ExceptFirst{{A}}$. %
We refer to ~${A}$ and the ~$X(i)$ collectively as a \emph{matched set} below. %
Define the \emph{\SuspendThingText{X(i)}{A}} denoted ~$\SuspendThing{X(i)}{A}$ as follows. %
Define set maps 
~$	{s}(i),{t}(i)
	\colon X(i)_{0}
	\rightarrow {A}
$ by 
~$	{s}(i)(y)=\Pred{i} 
$ 
and 
~$	{t}(i)(y)=i 
$ 
for  
~$	y\in X(i)_{0}
$ and for each ~$i\in\ExceptFirst{A}$. %
Then the suspension ~$\SuspendThing{X(i)}{A}$ is 
$$ \cdots 
\two/->`->/<500>^{\sum{s}_{1}(i)}
				_{{\sum{t}_{1}(i)}} 
\sum
	X(i)_{1}
\two/->`->/<500>^{\sum{s}_{0}(i)}
				_{\sum{t}_{0}(i)} 
\sum
	X(i)_{0}
\two/->`->/<500>^{\sum{s}(i)}
				_{\sum{t}(i)}
{A}
$$ 
where the coproducts are indexed over all ~
$i\in \ExceptFirst{A}$. 
Notice that 
$	{s}(i) s_{0}(i)
=	{s}(i) t_{0}(i)
$
as ~${s}(i)$ is constant. %
In addition as the required identities hold for each ~$X(i)$ then the universal property of  coproduct in ~$\CatName{Set}$ implies that \[
	\sum{s}(i)_{n} 
	\sum{s}(i)_{n+1}
	= 
	\sum{s}(i)_{n} 
	\sum{t}(i)_{n+1}
\] for ~$n\in\mathbb{N}_{+}$. %
Likewise for the target maps. %
Hence the source and target identities required of globular cardinals are satisfied by the suspension. %
The linear order given by the source and target maps is 
\[	\{
	x_{0}, X(1), x_{1}, X(2), x_{2}, 
	\ldots, x_{n-1}, X(n), x_{n}
	\}.
\]
Hence ~$\SuspendThing{X(i)}{A}$ is a globular cardinal. %

Similarly, we define a suspension operation on collections of morphisms of globular cardinals.
Let 
~$	X(i),A
$ 
and ~$Y(j),B$ be matched sets and let ~$X$ and ~$Y$ be their suspensions as defined above. %
Let 
~$	{f}\colon{A}\rightarrow{B}
$
be an incremental morphism of ordered sets and let 
~$	f(i) 
	\colon X(i) 
	\rightarrow Y({f} i) 
$ 
be a morphism of globular cardinals for each 
~$	i\in\ExceptFirst{A} 
$. %
Define the \emph{\SuspendThingText{f(i)}{f}} denoted ~$\SuspendThing{f(i)}{f}$ as 
$$\bfig
\square(500,0)|amma|
		/{@{>}@<3pt>}`{}`{->}`{@{>}@<3pt>}/
		[\cdots`X_2`\cdots`Y_2
		;\phantom{s}
			`\cdots
			`\sum f(i)_{1}
			`\phantom{s}
		]
\square(1000,0)|amma|
		/{@{>}@<3pt>}`{}`{->}`{@{>}@<3pt>}/
		[X_2`X_1`Y_2`Y_1
		;\phantom{s}
			`\sum f(i)_{1}
			`\sum f(i)_{0}
			`\phantom{s}
		]
\square(1500,0)|amma|
		/{@{>}@<3pt>}`{}`{->}`{@{>}@<3pt>}/
		[X_1`A`Y_1`\phantom{.}B
		;\phantom{s}
			`\sum f(i)_{0}
			`{f}
			`\phantom{s}
		]

\square(500,0)|bmmb|
		/{@{>}@<-3pt>}`{}`{->}`{@{>}@<-3pt>}/
		[\cdots`X_2`\cdots`Y_2
		;\phantom{t}
			`\cdots
			`\sum f(i)_{1}
			`\phantom{t}
		]
\square(1000,0)|bmmb|
		/{@{>}@<-3pt>}`{}`{->}`{@{>}@<-3pt>}/
		[X_2`X_1`Y_2`Y_1
		;\phantom{t}
			`\sum f(i)_{1}
			`\sum f(i)_{0}
			`\phantom{t}
		]
\square(1500,0)|bmmb|
		/{@{>}@<-3pt>}`{}`{->}`{@{>}@<-3pt>}/
		[X_1`A`Y_1`\phantom{.}B
		;\phantom{t}
			`\sum f(i)_{0}
			`{f}
			`\phantom{t}
		]
\efig$$
where the coproducts are indexed over ~
$\ExceptFirst{A}$. %
The squares commute by universal property of  coproduct in ~$\CatName{Set}$ and the suspension
~$	\SuspendThing{f(i)}{f}
$ 
is a morphism of globular cardinals. %
\end{definition}

%



	\subsection{Ordinal graphs}
	
	\label{sec: ordinal graphs}
	
The purpose of this section is to define the concept of an enriched graph called a  ~$\mathcal{V}\text{-graph}$ to allow an inductive definition equivalent to that of globular cardinals. %
An equivalence is demonstrated between the categories of
globular cardinals and ordinal graphs. %
The section closes with the definition of a map that 
sends objects of $\iPidi$ to ordinal graphs. 

	\begin{definition} 
	
A \emph{$\mathcal{V}\text{-graph}$} ~$\mathcal{G}$, for a category ~$\mathcal{V}$, consists of a set of vertices ~$\Root{\mathcal{G}}$ and an \emph{edge-object} ~$\mathcal{G}(x,y)$ in ~$\Obj\mathcal{V}$ for every pair of vertices ~$x,y$. %
A \emph{morphism 
of 
~$	\mathcal{V}\text{-graphs}
$} 
~$	f\colon\mathcal{G}\rightarrow\mathcal{H}
$ 
is a set map 
~$	f\colon\Root{\mathcal{G}}
	\rightarrow\Root{\mathcal{H}}
$
and for every edge-object 
~$	\mathcal{G}(x,y)
$ 
of ~$\mathcal{G}$ a morphism 
\[	f(x,y)
	\colon\mathcal{G}(x,y)
	\rightarrow\mathcal{H}(f x,f y)
\] 
of ~$\mathcal{V}$. %
We have the category ~$\mathcal{V}\text{-Gph}$ of ~$\mathcal{V}\text{-graphs}$, graphs enriched over ~$\mathcal{V}$.
\end{definition}

	\begin{definition} \label{def: ordinal graphs} 
	
Suppose ~$\mathcal{V}$ has initial object 0 and ~$\mathcal{U}$ is a subset of ~$\Obj\mathcal{V}$ containing 0. %
A \emph{~$\mathcal{U}$ ordinal ~$\mathcal{V}\text{-graph}$} ~$\mathcal{G}$ consists of a finite linearly ordered set ~$\Obj\mathcal{G}$, objects ~$\mathcal{G}(x,y)$ in ~$\mathcal{U}$ of ~$\mathcal{V}$ for all pairs $x, y$ in ~$\Obj\mathcal{G}$, and such that ~
$	\mathcal{G}(x,y) \not= 0$ if and only if
~$y$ is the successor of ~$x$. %
We have the category ~$(\mathcal{U},\mathcal{V})\text{-Gph}_\mathrm{ord}$ of ~$\mathcal{U}$ ordinal ~$\mathcal{V}\text{-graphs}$. %
Notice that the object maps are incremental. %
\end{definition}

	\begin{definition} 
	
We define two categories ~$\Graph_\mathbb{N}$ and ~$\OGraph$ of enriched graphs. %
Let ~$\Graph_0$ denote the category
~$	\{ \emptyset \}\text{-Gph}
$ 
and let 
~$	\Graph_{n+1}
$ 
denote the category
~$	\Graph_n\text{-Gph}
$ 
for each ~$n\in\mathbb{N}$. %
Define ~$\Graph_\mathbb{N}$ as the colimit of the diagram 
\[	\Graph_0
	\rightarrow \Graph_1
	\rightarrow \cdots 
	\rightarrow \Graph_n
	\rightarrow \cdots
\]
of inclusions. %
The empty graph ~$\emptyset$ has \emph{dimension} \TextMinusOne\ (minus one). %
A graph  of ~$\Graph_n$ has \emph{dimension} ~$n$.

We define ~$\OGraph$ the category of \emph{ordinal graphs}, a subcategory of ~$\Graph_\mathbb{N}$, which we demonstrate in Theorem \ref{thm: globcard to ograph} is equivalent to the category of globular cardinals. %
Let ~$\OGraph_0$ denote the category ~$( \emptyset,\emptyset ) \text{-Gph}_\mathrm{ord}$ and let ~$\OGraph_{n+1}$ denote the category ~$(\OGraph_n,\Graph_n)\text{-Gph}$ for ~$n\in\mathbb{N}$. %
Define ~$\OGraph$ as the colimit of the diagram
\[	\OGraph_{0}
	\rightarrow \OGraph_{1}
	\rightarrow \cdots 
	\rightarrow \OGraph_{n}
	\rightarrow \cdots
\]
of inclusions. %
\end{definition}

	\begin{definition}
	
	\label{def: functor globcard to ograph}
	
We define a functor 
\begin{align*}
	\FunctorGlobCardToOGraph
	\colon\GlobCard
	\rightarrow\OGraph
\end{align*} 
using induction on the dimension of globular cardinals. %
Let ~$X$ be a globular cardinal of dimension -1 (minus one). %
Then ~$X$ is empty and so an initial object. %
Define $\FunctorGlobCardToOGraph X$ as the empty ordinal graph. %
Assume ~$\FunctorGlobCardToOGraph$ is defined on globular cardinals of dimension $n$ and let ~$X$ have dimension ~$n+1$. %
Define 
~$	\FunctorGlobCardToOGraph X
$ as ~ $\mathcal{G}
$
where 
~$	\Root{\mathcal{G}}
=	X_0
$
and 
~$	\mathcal{G}(\Pred{x},x)
=	\FunctorGlobCardToOGraph 
		\SubThing{X}{\Pred{x},x}
$
for each $x\in\ExceptFirst{X_0}$. %

Define  
~$	\FunctorGlobCardToOGraph
$ 
on morphisms using induction on the dimension of their domain. %
Let 
~$	f\colon{X}\rightarrow{Y}
$ 
be a morphism with domain of dimension \TextMinusOne\ (minus one). %
Define 
~$	\FunctorGlobCardToOGraph f
$ as
~$	\emptyset 
	\rightarrow\FunctorGlobCardToOGraph{Y}
$ the unique morphism out of the empty graph. %
Assume ~$\FunctorGlobCardToOGraph$ is defined on morphisms with domain of dimension ~$n$ and let 
~$	f\colon{X}\rightarrow{Y}
$ 
be a morphism with domain of dimension ~$n+1$. %
Define 
~$	\FunctorGlobCardToOGraph f
$ as ~ $g$ with object map
~$	f_0
	\colon{X_0}
	\rightarrow{Y_0}
$ 
and 
~$	g(\Pred{x},x)
=	\FunctorGlobCardToOGraph f(\Pred{x},x)
$ for each ~$x \in\ExceptFirst{X_0}$. %
\end{definition}

	\begin{definition}
	
	\label{def: functor ograph to globcard}
	
We define a functor 
\begin{align*}
	\FunctorOGraphToGlobCard
	\colon\OGraph
	\rightarrow\GlobCard
\end{align*} 
using induction on the dimension of ordinal graphs. %
Let ~$\mathcal{G}$ be an ordinal graph of dimension \TextMinusOne\ (minus one). %
Then ~$\mathcal{G}$ is empty and so is initial. %
Define $\FunctorOGraphToGlobCard \mathcal{G}$ as the initial globular cardinal. %
Assume ~$\FunctorOGraphToGlobCard$ is defined on ordinal graphs of dimension $n$ and let $\mathcal{G}$ have dimension ~$n+1$. %
Define 
~$	\FunctorOGraphToGlobCard\mathcal{G}
$ as ~ 
$	\SuspendThing
			{\FunctorOGraphToGlobCard
				\mathcal{G}(\Pred{x},x)}
			{\Root{\mathcal{G}}}
$ where ~$\mathcal{G}(\Pred{x},x)$ is the collection of non-empty ordinal graphs of ~$\mathcal{G}$ indexed by ~$x\in\ExceptFirst{\Root{\mathcal{G}}}$. %

Define
~$	\FunctorOGraphToGlobCard
$ 
on morphisms using induction on the dimension of their domain. %
Let 
~$	g\colon{\mathcal{G}}\rightarrow{\mathcal{H}}
$ 
be a morphism with domain of dimension \TextMinusOne\ (minus one). %
Define 
~$	\FunctorOGraphToGlobCard g
$ as ~
$	
	\emptyset 
	\rightarrow\FunctorOGraphToGlobCard{\mathcal{H}}
$ the unique morphism out of the empty ordinal graph. %
Assume ~$\FunctorOGraphToGlobCard$ is defined on morphisms with domain of dimension ~$n$ and let 
~$	g
$ 
be a morphism with domain of dimension ~$n+1$. %
Define 
~$	\FunctorOGraphToGlobCard g
$ as the suspension
~$\SuspendThing
			{\FunctorOGraphToGlobCard 
				g(\Pred{x},x)
			}
			{g}
$ where ~$g(\Pred{x},x)$ is the collection of morphisms of ordinal graphs of ~$g$ indexed by ~
$x\in\ExceptFirst{\Root{\mathcal{G}}}$. %

\end{definition}

	\begin{theorem}
	
	\label{thm: globcard to ograph}
	
The category ~$\GlobCard$ is equivalent to the category ~$\OGraph$ by
\[	\FunctorGlobCardToOGraph \colon
		\GlobCard\rightarrow\OGraph
\]
and its equivalence inverse 
~$	\FunctorOGraphToGlobCard$. %
\end{theorem}

\begin{proof} 
We construct natural isomorphisms 
~$	\eta
	\colon\Id
	\Rightarrow
			\FunctorOGraphToGlobCard
			\FunctorGlobCardToOGraph
$ 
and
~$	\epsilon
	\colon
		\FunctorGlobCardToOGraph
		\FunctorOGraphToGlobCard
	\Rightarrow\Id
$ using induction on the dimension of globular cardinals (respectively ordinal graphs). %
Let ~$X$ be a globular cardinal of 
dimension \TextMinusOne\ (minus one). %
Then ~$X$ and ~$
	\FunctorOGraphToGlobCard
	\FunctorGlobCardToOGraph
	{X}
$ are both empty globular cardinals. %
Assume ~$\eta$ is a natural isomorphism for globular 
cardinals of dimension ~$n$ and let $X$ have 
dimension ~$n+1$. %
Let ~$
	X'
=	\FunctorOGraphToGlobCard
	\FunctorGlobCardToOGraph{X}
$ which is ~$
	\SuspendThing{\FunctorOGraphToGlobCard
					\FunctorGlobCardToOGraph
					 \SubThing{X}{\Pred{x},x}} {X_0}
$. %
We construct an isomorphism ~$
	f
	\colon{X}
	\rightarrow 
		\FunctorOGraphToGlobCard
		\FunctorGlobCardToOGraph{X}
$. %
Define ~$f_0$ as ~ $\Id_{X_0}$ and ~$f_n$ as the unique map out of the coproduct in 
$$\bfig
\square(0,0)|aaa|
	/->`->`->`-->/
	<1000,500>%
	[\FunctorOGraphToGlobCard
		\FunctorGlobCardToOGraph 
		X(\Pred{x},x)_n
	`X(\Pred{x},x)_n
	`\sum \FunctorOGraphToGlobCard
		\FunctorGlobCardToOGraph 
		X(\Pred{x},x)_n
	`X_{n+1}
	;\cong
	`\mathrm{copr}
	`\mathrm{incl}
	`f_n
	] 
\efig$$
for each ~$n\in\mathbb{N}_+$. %
Then ~$f$ is a bijection as the coprojections are monomorphisms and the composites with the inclusions are monomorphisms and are jointly epi. %
Naturality arises from the universal property of  coproduct. %

Let ~$\mathcal{G}$ be an ordinal graph of dimension -1 (minus one). %
Then ~$\mathcal{G}$ and ~$
	\FunctorGlobCardToOGraph
	\FunctorOGraphToGlobCard
	\mathcal{G}
$ 
are both empty ordinal graphs. %
Assume ~$\epsilon$ is a natural isomorphism for ordinal graphs of dimension ~$n$ and let $\mathcal{G}$ have dimension ~$n+1$. %
Define ~ $X$ as \[
	\FunctorOGraphToGlobCard
	{\mathcal{G}}
=	\SuspendThing{\FunctorOGraphToGlobCard
					\mathcal{G}(x_{i-1},x_i)
				}
				{\Root{\mathcal{G}}
				}. %
\]
The globular cardinal ~$\SubThing{X}{\Pred{x},x}$ is the largest subfunctor of ~$X$ satisfying the requirements of Definition ~\ref{def: restriction GC}. 
The coprojection ~\[
	\FunctorOGraphToGlobCard
					\mathcal{G}(\Pred{x},x)
	\rightarrow
	\sum\FunctorOGraphToGlobCard
					\mathcal{G}(\Pred{x},x)
\] is also such a functor and so is identically ~$\SubThing{X}{\Pred{x},x}$. %
Then 
~$	\mathcal{G}'
=	\FunctorGlobCardToOGraph X
=	\FunctorGlobCardToOGraph 
	\FunctorOGraphToGlobCard \mathcal{G}
$
has ~$\Root{\mathcal{G}'} = \Root{\mathcal{G}}$ 
and ~
$	\mathcal{G}'(\Pred{x},x) 
=
	\FunctorGlobCardToOGraph 
	\FunctorOGraphToGlobCard
	\mathcal{G}(\Pred{x},x)
$
which by induction is naturally isomorphic to 
~$	\mathcal{G}(\Pred{x},x)
$. %
Hence ~$\epsilon$ is a natural isomorphism since for any ordinal graph ~$\mathcal{G}$ then ~$\epsilon_\mathcal{G}$ consists of an identity and components which are natural isomorphisms.
\end{proof}

	\begin{definition}
	
	\label{def: pidi to ograph}
		
We define a map 
\[	\ObjMapiPidiToOGraph
	\colon\Obj\iPidi
	\rightarrow\Obj\OGraph
\]
by induction on the height of objects of $\iPidi$. %
Send the initial object ~$[-1]$ to the empty ordinal
graph. %
Assume ~$\ObjMapiPidiToOGraph$ is defined on objects
of height ~$n$ and let ~ $H$ have height ~$n+1$. %
Let ~$\ObjMapiPidiToOGraph H$ be the ordinal 
graph ~$	\mathcal{G}$ with ~$
	\Root{\mathcal{G}}=\Root{H}
$ and \[
	\mathcal{G}({i-1},{i})
=	\ObjMapiPidiToOGraph H(i)
\] for ~$i$ which are not endpoints of ~$
	(\Root{H})^{\wedge}
$. %
Recall ~$H(i)$ is trivial when ~$i$ is an endpoint of ~$(\Root{H})^\wedge$. %
Notice that objects of $\iPidi$ 
of height ~$n$ are sent to ordinal graphs of dimension ~$n-1$. %
	
Define 
~$	\ObjMapiPidiToOGraph'
	\colon\Obj\OGraph
	\rightarrow\Obj\iPidi
$ 
by induction on the dimension of ordinal graphs. 
Send the initial ordinal graph to the trivial object ~$[-1]$. 
Assume ~$\ObjMapiPidiToOGraph'$ is defined on ordinal graphs  of dimension ~$n$ and let ~$\mathcal{G}$ have dimension ~$n+1$. %
Suppose ~ 
$	\Root{\mathcal{G}} = \{x_0,x_1,\ldots,x_p\}
$. %
Define
~$	\ObjMapiPidiToOGraph'\mathcal{G}
$ 
as the object 
~$H$ with~ 
$	\Root{H}=[p]
$, 
with ~$H(0)$ and ~$H(p+1)$ the trivial object of ~$\iPidi$ 
and with 
~$	H(i) = \ObjMapiPidiToOGraph' 
			\mathcal{G}(x_{i-1},x_i))
$
for each ~$x\in\ExceptFirst{[p]}$. %
Notice that the ordinal graphs of dimension ~$n$ are sent to objects of $\iPidi$ 
of height ~$n+1$. %

\end{definition}

	\begin{theorem}
	
	\label{theorem: iPidi to OGraph}

The object map
~$	\ObjMapiPidiToOGraph
$ 
is surjective. 
\end{theorem}

	\begin{proof}
	
Clearly $\ObjMapiPidiToOGraph$ is surjective on 
ordinal graphs of dimension \TextMinusOne\ (minus one). %
Assume that ~$\ObjMapiPidiToOGraph$ is surjective on 
ordinal graphs of dimension ~$n$, let $\mathcal{G}$ 
be an ordinal graph of dimension ~$n+1$ and let $
	\mathcal{G}{'} 
=	\ObjMapiPidiToOGraph
	\ObjMapiPidiToOGraph{'}
	\mathcal{G}
$. %
Then $
	\Obj\mathcal{G} 
=	\Obj\mathcal{G}{'}
$ and, by induction, 
for each $
	i\in\ExceptFirst{\Obj\mathcal{G}}
$ we have
$ 
	\SubThing{\mathcal{G}}{\Pred{i},i}
=	\SubThing{\mathcal{G}{'}}{\Pred{i},i}
$. %
Hence $
	\mathcal{G}=\mathcal{G}{'}
$ and $\ObjMapiPidiToOGraph$ is surjective.
\end{proof}

%% file: omegacats.tex
	\subsection{$\omega\text{-categories}$}
	
	\label{sec: omega cats}

We start by recalling the definition 
of \TextCategory{\omega}. %
Definition \ref{def: glob free} is adapted from Michael Batanin's construction in \cite{MB_MGC} of the free \TextCategory{\omega} on a globular set.
In addition we define the free \TextCategory{\omega} 
on an ordinal graph and show that it is isomorphic 
to the free \TextCategory{\omega} of the corresponding
globular cardinal. 

	\begin{definition} \label{def: omega category}
	
A \emph{\TextCategory{1}} is an ordinary category and a \emph{\TextCategory{0}} is a discrete \TextCategory{1}. %
An \emph{\TextCategory{n}} is a category enriched 
over 
\TextCategories{(n-1)} 
for ~$n\in\mathbb{N}_+$. %
Then ~$\wCat$, the category of \emph{\TextCategories{\omega}}, is the colimit of the diagram 
\[	\nCat{0} 
	\rightarrow \nCat{1} 
	\rightarrow \nCat{2} 
	\rightarrow \cdots 
	\rightarrow \nCat{n} 
	\rightarrow \cdots
\]
of inclusions. %
The initial \TextCategory{\omega} has \emph{dimension} \TextMinusOne\ (minus one). %
The \TextCategories{\omega} of ~$\nCat{n}$ have \emph{dimension} ~$n$. %
\end{definition}

	\begin{definition}
	
Let ~$Y$ be a globular cardinal with  ~$x$ and ~$y$ consecutive \MathText{n}{-vertices}. %
Define
$	\CompThing{Y}{x,y}
$ as the largest subfunctor of ~$Y$ such that ~
$	\CompThing{Y}{x,y}_n=\{x,y\}
$. %
We have the inclusion 
~$	\CompThing{\iota}{x,y}
	\colon
	\CompThing{Y}{x,y}
	\rightarrow{Y}
$. %
Let ~ $\gamma$ be a morphism of globular cardinals. %
Define
$	\CompThing{\gamma}{x,y}
$ as the composite ~
$	\gamma\circ\CompThing{\iota}{x,y}
$. 
\end{definition}

	\begin{definition} \label{def: glob free}
	
We define a functor
\begin{align*}
	\FreeFunctor
	\colon\GlobCard
	\rightarrow\wCat. 
\end{align*} %
Let ~$X$ be a globular cardinal. %
The ~\MathText{n}{-cells} of ~$\FreeFunctor{X}$ 
are isomorphism classes of objects of ~ 
$	\GlobCard/X
$ which are globular morphisms ~
$	\gamma\colon Y\rightarrow{X}
$ where ~ $Y$ has dimension ~$n$. %

Let ~$Y$ be a globular cardinal. %
We define the ~\MathText{m}{-source} of ~$Y$ denoted ~$\GCsrc{m} Y$ (respectively the ~\MathText{m}{-target} of ~$Y$ denoted ~$\GCtgt{m} Y$) for ~$m\in\mathbb{N}$. %
Define ~$\GCsrc{0} Y$ (respectively ~$\GCtgt{0} Y$) as the smallest subfunctor of ~$Y$ containing the least (respectively greatest) element of ~$Y$. %
For ~$m\ge 1$ define ~$\GCsrc{m} Y$ 
(respectively ~$\GCtgt{m} Y$) as the smallest subfunctor of ~$Y$ containing ~$Y_\ell$ for all 
~$\ell<m$ and containing the least (respectively greatest) element of ~
$	\CompThing{Y}{\Pred{y},y}_m
$ for all ~$y$ in ~
$	\ExceptFirst{Y_{m-1}}
$. %
Given an \MathText{n}{-cell} ~ $\gamma$
(representing an isomorphism class) 
then ~$\GCdom{m}(\gamma)$ is (the isomorphism
class of) the composite
\[
\GCsrc{m} Y
\to/->/<500>^{\mathrm{incl}}
Y
\to/->/<500>^{\gamma}
X.
\] %
Likewise for ~$\GCcod{n}(\gamma)$. %

Composition in ~$\FreeFunctor{X}$ is given by pushout. %
Let ~ 
$\alpha\colon{Y}\rightarrow{X}$ 
and ~
$\beta\colon{Z}\rightarrow{X}$ 
be
\MathText{n}{-cells} with ~ 
$	\GCdom{m}(\beta)=\GCcod{m}(\alpha)
$. %
There is a unique isomorphism ~ 
$	\delta
	\colon\GCsrc{m}Z\rightarrow\GCtgt{m}Y
$ such that ~ 
$	\GCdom{m}\beta=\GCcod{m}\alpha\circ\delta
$ where composition and equality is of globular morphisms. %
Define  ~ 
$	\beta\nComp{m}{}\alpha$ 
as the unique morphism 
in $$\bfig
\node TL(0,0)[\GCsrc{m}Z]
\node TR(600,0)[Z]
\node BL(0,-600)[Y]
\node BR(600,-600)[P]
\node X(850,-850)[X]
\arrow/->/[TL`TR;\mathrm{incl}]
\arrow/->/[TL`BL;\delta]
\arrow/->/[TR`BR;]
\arrow/->/[BL`BR;]
\arrow|b|/{@{->}@/_5pt/}/[BL`X;\alpha]
\arrow/{@{->}@/^5pt/}/[TR`X;\beta]
\arrow/-->/[BR`X;]
\efig
$$
out of the pushout ~$P$ which is defined as ~ 
$	Z\backslash\GCsrc{m}Z + Y
$. %
Specifically, we have ~ $P_\ell=Y_\ell=Z_\ell$ 
for ~ $\ell<m$, ~ 
$	P_m=Z_m\backslash(\GCsrc{m}Z)_m + Y_m
$ and ~ 
$	P_\ell=Y_\ell + Z_\ell
$ for ~ $\ell>m$ where the ~$+$ operation is the 
ordered union of linearly ordered sets. %
Let ~ $\circ$ denote \MathText{0}{-composition}. %

An \MathText{n}{-cell} $\gamma$ is 
\emph{indecomposable} when the cardinality of ~
$	\CompThing{Y}{\Pred{y},y}_m
$ is 2 (two) for all ~ $m<n$ and is 1 (one) 
when ~ $m=n$.
An \MathText{n}{-cell} $\gamma$ is 
\emph{\MathText{0}{-indecomposable}} when the
cardinality of ~ 
$	Y_0
$ is less than or equal to 2 (two). %
An \MathText{n}{-cell} $\gamma$ is 
\emph{\MathText{m}{-indecomposable}} 
(for ~ $m\ge 1$) when the
cardinality of ~ 
$	\CompThing{Y}{\Pred{y},y}_m
$ is less than or equal to 2 (two) for all ~
$	y\in \ExceptFirst{Y_{m-1}}
$. %
\end{definition}

Observations \ref{obs: decomp}, \ref{obs: dom sub} and \ref{obs: comp sub} are referred to in Definition \ref{def: two free functors}. 

	\begin{observation} \label{obs: decomp}

We identify an arbitrary 
\MathText{n}{-cell}  of ~ $\FreeFunctor{X}$ 
with a canonical \MathText{0}{-composition} 
of \MathText{0}{-indecomposable} 
\MathText{n}{-cells}. %
Let ~$\gamma\colon{Y}\rightarrow{X}$ be an  
\MathText{n}{-cell} where ~$Y_0$ is ~$
	\{y_0, y_1,\ldots,y_p\}
$. %
We have
\[	\gamma
=	\CompThing{\gamma}{y_{p-1},y_p}
	\circ \ldots \circ
	\CompThing{\gamma}{y_0,y_1}. %
\]
We denote 
this composite as ~ $\DeComp{\gamma}{y}$. %
and understand that ~ $y$
is an index over ~ $\ExceptFirst{Y_0}
$. %
\end{observation}

	\begin{observation} \label{obs: dom sub} 

We show 
$	
		\SubThing
			{(\GCdom{m}\gamma)}
			{\Pred{y},y}
=	\GCdom{m-1}
		\SubThing{\gamma}{\Pred{y},y}
$
for an \MathText{n}{-cell} ~ $\gamma$ of ~
$	\FreeFunctor{X}$ where ~
$	m<n$ and ~ $y\in\ExceptFirst{Y_0}$. %
The lifting in 
\[\bfig
\hSquares|aaaaaab|/-->`-->`>`>`>`>`>/%
[\SubThing{(\GCsrc{m} Y)}{\Pred{y},y}
`\SubThing{Y}{\Pred{y},y}
`\SubThing{X}{\Pred{\gamma y},\gamma y}
`\GCsrc{m} {Y}\circ u_1
`Y\circ u_1
`X\circ u_1
;
`\SubThing{\gamma}{\Pred{y},y}
`
`
`
`
`\gamma\cdot u_1
]
\efig\]
along the vertical arrows is ~ 
$	\SubThing{(\GCdom{m}\gamma)}{\Pred{y},y}
$
where the inclusions are described in Definition 
\ref{def: restriction GC}. %
The composite 
\[
\GCsrc{m-1} \SubThing{Y}{\Pred{y},y}
\to/->/<500>^{\mathrm{incl}}
\SubThing{Y}{\Pred{y},y}
\to/->/<500>^{\SubThing{\gamma}{\Pred{y},y}}
\SubThing{X}{\Pred{y},y}
\] 
is ~ $\GCdom{m-1}
		\SubThing{\gamma}{\Pred{y},y}
$. %
It remains to show that 
$	\SubThing	{\left( \GCsrc{m} {Y}\right)}
				{\Pred{y},y}
$ and ~ 
$	\GCsrc{m-1} \left(
				\SubThing{Y}{\Pred{y},y}
				\right)
$ are identical. %
Let ~ $x$ be an element of ~ 
$	\SubThing	{\left( \GCsrc{m} {Y}\right)}
				{\Pred{y},y}
$. %
Then ~ $\Pred{y} \GCrel x \GCrel y$ and either 
\[
	x\in Y_\ell 
	\text{ for } 1<\ell<m
\]
or 
\[
	x \text{ is the least element of }
	\CompThing{Y}{\Pred{z},z}_{m}
	\text{ for some }
	z\in \ExceptFirst{Y_{m-1}}. %
\] 
Given the first condition then they can
be rewritten as 
\[
	x\in\SubThing{Y}{\Pred{y},y}_\ell 
	\text{ for } \ell<m-1
\]
or 
\[
	x \text{ is the least element of }
	\CompThing{\SubThing{Y}{\Pred{y},y}}
	{\Pred{z},z}_{\ell}
	\text{ for some }
	z\in\ExceptFirst{
		\SubThing{Y}{\Pred{y},y}_{m-2}
		}. %
\] 
Jointly these rewritten conditions state that ~$x$ 
is in $
	\GCsrc{m-1} \left(
				\SubThing{Y}{\Pred{y},y}
				\right)
$. %
The converse is also true. %
We have that ~$	
		\SubThing
			{(\GCdom{m}\gamma)}
			{\Pred{y},y}
$ and ~$
	\GCdom{m-1}
		\SubThing{\gamma}{\Pred{y},y}
$ are identical. 
\end{observation}

	\begin{observation} \label{obs: comp sub} 
	
We show 
$	\SubThing
	{(\beta\nComp{m}{}\alpha)}
	{\Pred{y},y}
=	\SubThing{\beta}{\Pred{y},y}
	\nComp{m-1}{}
	\SubThing{\alpha}{\Pred{y},y}
$
for \MathText{m}{-composable} \MathText{n}{-cells} ~ $\alpha$ and ~ $\beta$ of ~
$	\FreeFunctor{X}$ where ~
$	m<n$ and ~ $y\in\ExceptFirst{Y_0}$. %
Let ~ 
$\alpha\colon{Y}\rightarrow{X}$ 
and ~
$\beta\colon{Z}\rightarrow{X}$ 
be \MathText{m}{-composable} 
\MathText{n}{-cells}. %
Their composition ~
$	\beta\nComp{m}{}\alpha$ 
is given the unique morphism out of the pushout in
$$\bfig
\node TL(0,0)[\GCsrc{m}Z]
\node TR(600,0)[Z]
\node BL(0,-600)[Y]
\node BR(600,-600)[P]
\node X(850,-850)[X]
\arrow/->/[TL`TR;\mathrm{incl}]
\arrow/->/[TL`BL;\delta]
\arrow/->/[TR`BR;]
\arrow/->/[BL`BR;]
\arrow|b|/{@{->}@/_5pt/}/[BL`X;\alpha]
\arrow/{@{->}@/^5pt/}/[TR`X;\beta]
\arrow/-->/[BR`X;]
\efig
$$
where ~ 
$	P=Z\backslash\GCsrc{m}Z + Y
$. %
Let ~ $y$ be an element of ~ $P$. %
By Observation \ref{obs: dom sub} then ~ 
$	\SubThing{(\GCsrc{m} Z)}{\Pred{y},y}
$ and ~
$	\GCsrc{m-1} \SubThing{Z}{\Pred{y},y}
$ are identical. %
Note that ~ 
$	\SubThing{Z}{{\Pred{y},y}}
	\backslash\SubThing{(\GCsrc{m}Z)}{\Pred{y},y}
$ is
$	\SubThing{Z}{{\Pred{y},y}}
	\backslash\GCsrc{m}Z
$
and further that
$	\SubThing{Z}{\Pred{y},y}
				\backslash\GCsrc{m}Z 
				+ \SubThing{Y}{\Pred{y},y}
$ is
$	\SubThing
		{(Z\backslash\GCsrc{m}Z + Y)}
		{\Pred{y},y}
$. %
Then ~ $\SubThing{P}{\Pred{y},y}$ is the 
pushout of 
$$\bfig
\node TL(0,0)[\SubThing{\GCsrc{m-1}Z}{\Pred{y},y}]
\node TR(600,0)[\SubThing{Z}{\Pred{y},y}]
\node BL(0,-600)[\SubThing{Y}{\Pred{y},y}]
\node BR(600,-600)[\SubThing{P}{\Pred{y},y}]
\node X(850,-850)[\SubThing{X}{\Pred{y},y}]
\arrow/->/[TL`TR;]
\arrow/->/[TL`BL;
					]
\arrow/->/[TR`BR;]
\arrow/->/[BL`BR;]
\arrow|b|/{@{->}@/_5pt/}/[BL`X;\SubThing{\alpha}{\Pred{y},y}]
\arrow/{@{->}@/^5pt/}/[TR`X;\SubThing{\beta}{\Pred{y},y}]
\arrow/-->/[BR`X;]
\efig
$$
which is the lifting of the above diagram 
determined by the element ~ $y\in P$. %
As all horizontal and vertical maps above are 
monomorphisms we avoid unnecessary notation
by labeling all restrictions with ~ 
$(\Pred{y},y)$. %
Hence the unique map of the second pushout is ~ 
$	\SubThing{(\beta\nComp{m}{}\alpha)}{\Pred{y},y}
$ from the lifting and 
$	\SubThing{\beta}{\Pred{y},y}
	\nComp{m}{}
	\SubThing{\alpha}{\Pred{y},y}
$ by definition. 
\end{observation}

	\begin{definition} \label{def: ordinal free}
	
We define a functor 
\begin{align*}
	\mathfrak{F}
	\colon\Graph_\mathbb{N}
	\rightarrow\wCat
\end{align*}
inductively on the dimension of ordinal graphs. %
Let ~$\mathcal{G}$ be the enriched 
graph of dimension \TextMinusOne\ (minus one) which is the empty graph. %
Define ~$\FreeFunctor\mathcal{G}$ as the empty \TextCategory{\omega}. %
Assume ~$\FreeFunctor$ is defined on enriched graphs of dimension ~$n$ and let ~$\mathcal{G}$ have dimension ~$n+1$.  
Define the object set of ~$\mathfrak{F}\mathcal{G}$ as the object set of ~$\mathcal{G}$.
Its hom-sets are defined by induction for distinct objects ~$x$ and ~$y$ as
\[	\left(\mathfrak{F}\mathcal{G}\right)\left(x,y\right) = 
	\sum_{x_0, ..., x_n \in \Obj \mathcal{G}}
	\mathfrak{F}\left(
				\mathcal{G}
				\left(x_{n-1},x_n
				\right)\right)
	\times ... \times
	\mathfrak{F}\left(\mathcal{G}\left(x_{0},x_1\right)\right)
\]
where ~$x=x_0$ and ~$y=x_n$. 
For all ~$x\in\Obj\mathcal{G}$ then ~$(\mathfrak{F}\mathcal{G})(x,x)$ is defined as 
\[	\mathcal{C}_T 
	+ \sum_{x_0, ..., x_n \in \Obj\mathcal{G}}
	\mathfrak{F}(\mathcal{G}(x_{n-1},x_n))
	\times ... \times
	\mathfrak{F}(\mathcal{G}(x_{0},x_1))
\]
where ~$\mathcal{C}_T$ is the terminal \TextCategory{\omega}. %

Define ~$\FreeFunctor$ on morphisms as follows. %
Let ~
$	g\colon\mathcal{G}\rightarrow\mathcal{H}
$ be a morphism of enriched graphs with domain of dimension \TextMinusOne\ (minus one). %
Define ~$\FreeFunctor{g}$ as the unique \TextFunctor{\omega}  ~
$	\emptyset\rightarrow\FreeFunctor\mathcal{H}
$. %
Assume ~$\FreeFunctor$ is defined for morphisms 
with domain of dimension ~$n$ and let ~
$	g\colon\mathcal{G}\rightarrow\mathcal{H}
$ be a morphism of enriched graphs with domain of dimension ~$n+1$. %
Let ~$x_0, x_1,\ldots,x_p$ denote the objects of ~$\mathcal{G}$. %
The morphisms of hom-objects are defined for distinct objects ~$x$ and ~$y$ as
\[	\left(\mathfrak{F}{g}\right)
		_{x,y} = 
	\sum_{x_0, ..., x_n \in \Obj\mathcal{G}}
	\mathfrak{F}\left(g\left(x_{n-1},x_n\right)\right)
	\times ... \times
	\mathfrak{F}\left(g\left(x_{0},x_1\right)\right)
\]
where ~$x=x_0$ and ~$y=x_n$. 
For all ~$x\in\Obj\mathcal{G}$ then ~
$	(\mathfrak{F}{g})_{x,x} 
$ is defined as 
\[	{g}_T^{\phantom{X}} 
	+ \sum_{x_0, ..., x_n \in \Obj\mathcal{G}}
	\mathfrak{F}(g(x_{n-1},x_n))
	\times ... \times
	\mathfrak{F}(g(x_{0},x_1))
\]
where ~
$	g^{\phantom{X}}_T
	\colon\mathcal{C}^{\phantom{X}}_T
	\rightarrow\mathcal{C}^{\phantom{X}}_T
$. %

\end{definition}

	\begin{definition}

We define a forgetful functor 
\[	\ForgetfulFunctor
	\colon\wCat
	\rightarrow\Graph_\mathbb{N}
\] 
using induction on the dimension of \TextCategories{\omega}. %
Let ~ $\mathcal{C}$ be the \TextCategory{\omega} of dimension \TextMinusOne\ (minus one) which is the empty \TextCategory{\omega}. %
Define ~ $\ForgetfulFunctor\mathcal{C}$ as the empty ordinal graph. %
Assume ~ $\ForgetfulFunctor$ is defined on \TextCategories{\omega} of dimension ~ $n$ and let ~ $\mathcal{C}$ have dimension ~ $n+1$. %
Define ~ $\ForgetfulFunctor\mathcal{C}$ as the ordinal graph with object set ~ $\Obj\mathcal{C}$ and with edge-object ~ 
$	(\ForgetfulFunctor\mathcal{C})(x,y)
=	\ForgetfulFunctor(\mathcal{C}(x,y))
$ determined by induction for each pair of objects ~ $x$ and ~$y$. %

Define ~ $\ForgetfulFunctor$ on morphisms as follows. %
Let ~ 
$	\mathcal{F}\colon\mathcal{C}\rightarrow\mathcal{A}
$ be an \TextFunctor{\omega} with domain of dimension \TextMinusOne\ (minus one). %
Define ~ $\ForgetfulFunctor\mathcal{F}$ as the unique morphism ~
$	\emptyset\rightarrow
	\ForgetfulFunctor\mathcal{A}
$ of enriched graphs. %
Assume ~ $\ForgetfulFunctor$ is defined on \TextFunctors{\omega} with domain of dimension ~ $n$ and let ~ $\mathcal{F}$ have domain of dimension ~ $n+1$. %
Define ~ $\ForgetfulFunctor\mathcal{F}$ as the ordinal graph with object set morphism identical with that of ~ $\mathcal{F}$ and with edge-object morphism ~ 
$	(\ForgetfulFunctor\mathcal{F})(x,y)
=	\ForgetfulFunctor(\mathcal{F}(x,y))
$ determined by induction for each pair of objects ~ $x$ and ~$y$. %
\end{definition}

	\begin{theorem}
	
The functor ~ $\FreeFunctor$ is left adjoint to ~ 
$\ForgetfulFunctor$. 
\end{theorem}

	\begin{proof}

We define, given an ordinal graph ~ $\mathcal{G}$ and an \TextCategory{\omega} ~ $\mathcal{C}$, a bijection 
\[	\phi
	\colon\Graph_\mathbb{N}
		(\mathcal{G},\ForgetfulFunctor\mathcal{C})
	\rightarrow\wCat(\FreeFunctor\mathcal{G},\mathcal{C})
\]
using induction on the dimension of ordinal graphs. %
Let ~ 
$	g	\colon\mathcal{G}
		\rightarrow\ForgetfulFunctor\mathcal{C}
$ be a morphism of ordinal graphs with domain of dimension \TextMinusOne\ (minus one). %
Define ~ $\phi g$ as the unique \TextFunctor{\omega} ~ 
$	\FreeFunctor\mathcal{G}\rightarrow\mathcal{C}
$. %
Assume ~ $\phi$ is defined and is a bijection on morphisms of ordinal graphs with domain of dimension  ~ $n$ and let ~ $g$ be such a morphism with domain of dimension  ~ $n+1$. %
Define ~ 
$	\phi g 
= 	\mathcal{F}
	\colon\FreeFunctor\mathcal{G}
	\rightarrow\mathcal{C}
$ as follows. %
The object morphism of ~ $\mathcal{F}$ is that of ~ $g$. %
Define, using the induction assumption, the morphism ~ $\mathcal{F}(x,y)$ of hom-sets as  ~ $\phi g(x,y)$ for each pair of objects ~ $x,y$ of ~ $\mathcal{G}$. %

Suppose that ~ $g$ and ~ $g'$ are morphisms of ordinal graphs such that ~ $\phi g = \phi g'$. %
Then the object maps of ~$g$ and ~$g'$ are identical and by the induction assumption the morphisms of edge-objects are identical. %
Hence ~ $\phi$ is injective. %

Let ~
$	\mathcal{F}
	\colon\FreeFunctor\mathcal{G}
	\rightarrow\mathcal{C}
$ be an \TextFunctor{\omega}. %
Define a morphism ~ $g$ of ordinal graphs as follows. %
The object map of ~ $g$ is that of ~$F$. The morphisms of edge-objects are defined using the induction assumption. %
Hence  ~$\phi$ is surjective and ~ $\FreeFunctor$ is left adjoint to ~ $\ForgetfulFunctor$. %
\end{proof}

	\begin{observation}

	\label{obs: restrict free}
	
In the sequel we restrict ~$\mathfrak{F}$ of 
Definition \ref{def: ordinal free} to the 
category of ordinal graphs. %
Recall that an ordinal graph ~$\mathcal{G}$ has ~$\mathcal{G}(x,y)$ non-empty if and only if ~$y$ is the successor of ~$x$. %
Let ~$\mathcal{G}$ be an ordinal graph. %
For all objects ~$x$ of ~$\mathcal{G}$ we have the hom-object ~
$	\left(\mathfrak{F}\mathcal{G}\right)
		\left(x,x\right) = 
	\mathcal{C}_T
$. %
For every pair ~$x,y$ of distinct objects then
\[	\left(\mathfrak{F}\mathcal{G}\right)
		_{x,y} = 
	\mathfrak{F}\left(\mathcal{G}
			\left(x_{n-1},x_n\right)\right)
	\times ... \times
	\mathfrak{F}\left(\mathcal{G}
			\left(x_{0},x_1\right)\right).
\]
Given a morphism
~$	g\colon\mathcal{G}\rightarrow\mathcal{H}
$
we have 
\[	\left(\mathfrak{F}g\right)
		_{x,y} = 
	\mathfrak{F}
		\left(g\left(x_{n-1},x_n\right)\right)
	\times ... \times
	\mathfrak{F}
		\left(g\left(x_{0},x_1\right)\right)
\]
where ~$x=x_0$, ~$y=x_n$ and ~$x_{i}$ is the successor of ~$x_{i-1}$ for ~$i=1,..,n$. %
\end{observation}

	\begin{observation}
	
	\label{obs: enriched one type}
	
We describe here the ~\MathText{n}{-cells}, domain, codomain and composition operations of ~$\FreeFunctor\mathcal{G}$ for an ordinal graph ~$\mathcal{G}$ with objects ~$x_0,x_1,\ldots,x_p$. %
A \MathText{0}{-cell} of ~
$	\FreeFunctor\mathcal{G}
$ is an object of ~$\mathcal{G}$. %
An \MathText{n}{-cell} ~$y$ of 
~$	\FreeFunctor\mathcal{G}
$ is a sequence 
~$	(y_i)_{i=h+1}^{k}
$ 
of ~$(n-1)\text{-cells}$, one from each factor, of the product
~$	\prod_{j=h+1}^{k}
		\FreeFunctor
			\left(\mathcal{G}(x_{i-1},x_i) 
			\right)
$ 
for ~$h\le k$ in ~$\{0,1,\ldots,p\}$. %
The ~\MathText{0}{-domain} of ~$y$ is ~$x_h$ and the ~\MathText{0}{-codomain} of ~$y$ is ~$x_k$. %
Composition (\MathText{0}{-composition}) is denoted by ~$\circ$ and is given by concatenation of sequences. %
Hence every \MathText{n}{-cell} is identified 
with a unique  \MathText{0}{-composition}. %

The \MathText{m}{-domain} 
and \MathText{m}{-codomain} 
of an \MathText{n}{-cell} ~$y$ denoted ~
$\GCdom{m}y$ and ~$\GCcod{m}y$ are the 
\MathText{0}{-compositions}
~$	\circ_{i=h+1}^{k} \GCdom{m-1} y_i
$ and 
~$	\circ_{i=h+1}^{k} \GCcod{m-1} y_i
$ (respectively) for ~$m<n$. %
The \MathText{m}{-composition} of \MathText{n}{-cells} ~$y$ and ~$z$ is defined
\[	\left(y_i\right)_{i=h+1}^{k}
	\nComp{m}{}
	\left(z_i\right)_{i=h+1}^{k}
\quad=\quad
	\circ_{i=h+1}^{k}
	\left(
		y_i \nComp{m-1}{} z_i
	\right)
\]
for ~$m<n$ where the \MathText{(m-1)}{-composition} is in ~
$	\FreeFunctor\mathcal{G}(x_{i-1},x_i)
$. %
\end{observation}

	\begin{definition}
	
	\label{def: two free functors}

We have two \emph{free} functors both denoted ~$\FreeFunctor$, one for globular cardinals and one for ordinal graphs and the functor $
	\FunctorGlobCardToOGraph
	\colon\GlobCard
	\rightarrow\OGraph
$ of Definition \ref{def: functor globcard to ograph}. %
Let ~ $X$ be a globular cardinal. %
We define an \TextFunctor{\omega}
\[	\FreeIsom\colon\FreeFunctor{X}
	\rightarrow\FreeFunctor
			\FunctorGlobCardToOGraph{X}
\]
using induction on the dimension of globular 
cardinals. %

Let ~$X$ be a globular cardinal of dimension \TextMinusOne\ (minus one). %
Then ~$\FreeFunctor{X}$ and ~$\FreeFunctor\FunctorGlobCardToOGraph{X}$ are both the empty \TextCategory{\omega}. %
Assume that ~$\FreeIsom$ is defined
for globular cardinals of dimension ~$n$ and let ~$X$ have dimension ~$n+1$. %
Define ~$\FreeIsom$ by induction on \MathText{m}{-cells}. %
Let 
~$	\gamma\colon{Y}\rightarrow{X}
$  be a \MathText{0}{-cell} of ~$\FreeFunctor{X}
$. %
Define ~$\FreeIsom$ on ~$\gamma$ as ~$\gamma_0(y)$
where ~$y$ is the unique element of ~$Y$ 
(and so of ~$Y_0$). %
Assume ~$L$ is defined on \MathText{m}{-cells} and let 
~$	\gamma
$ be an ~\MathText{(m+1)}{-cell} of 
~$	\FreeFunctor{X}$. %
By Observation \ref{obs: decomp} we have ~
$	\DeComp{\gamma}{y}
$ the unique \MathText{0}{-decomposition} of ~
$\gamma$. %
Using the induction assumption define ~
$\FreeIsom\gamma$ as ~ 
\[	\left(
		\FreeIsom\SubThing{\gamma}{\Pred{y},y}
	\right)_{y\in \ExceptFirst{Y_0}}
\quad\text{in}\quad
\prod_{y\in \gamma(\ExceptFirst{Y_0})}
	\FreeFunctor
	\left(
	\FunctorGlobCardToOGraph{X}
		(\gamma\Pred{y},\gamma y)
	\right).
\] %
Note that ~ $\gamma$ is an \MathText{m}{-cell}
if and only if ~ $\FreeIsom\gamma$ is. %

We show using induction that ~ $\FreeIsom$ preserves
the \MathText{\ell}{-domain} 
and \MathText{\ell}{-codomain} operations
for ~ $\ell<m$. %
The \MathText{0}{-domain} of ~ $\FreeIsom\gamma$
is ~ $\gamma y_0$ and the \MathText{0}{-codomain} 
is ~ $\gamma y_\ell$ 
where ~ $y_0$ is the first element, 
respectively ~ $y_\ell$ is the last element, 
of ~ $Y_0$. %
Hence ~ $\FreeIsom$ preserves ~ $\GCdom{0}$ 
and ~ $\GCcod{0}$. %

Assume that ~ $\FreeIsom$ preserves the
\MathText{\ell}{-domain} operation. %
The following derivation begins by replacing ~
$\gamma$ with its 
\MathText{0}{-decomposition}, uses 
a basic property of ~ \TextCategories{\omega} 
to proceed from line 1 to line 2, uses 
Observation \ref{obs: dom sub} to proceed from line 3 to line 4 and uses the induction assumption to proceed from line 4 to line 5. %
We have
\begin{align*}
\FreeIsom\GCdom{\ell+1}\gamma
	& =	\FreeIsom\GCdom{\ell+1}\DeComp{\gamma}{y}
\\	& =	\FreeIsom\DeComp{\GCdom{\ell+1}\gamma}{y} 
\\	& =	\left(
			\FreeIsom
			\SubThing
			 {\left(\GCdom{\ell+1}
			 	\CompThing{\gamma}{\Pred{y},y}
			 \right)}
			{py,y}
		\right)_y
\\	& =	\left(
			\FreeIsom
			\GCdom{\ell}\left(\SubThing
			 {\CompThing{\gamma}{\Pred{y},y}}
			{py,y}\right)
		\right)_y
\\	& =	\left(
			\GCdom{\ell}\FreeIsom
			\SubThing
			 {\CompThing{\gamma}{\Pred{y},y}}
			{py,y}
		\right)_y
\\	& =	\GCdom{\ell+1}
		\left(
			\FreeIsom
			\SubThing
			 {\gamma}
			{py,y}
		\right)_y
\\	& =	\GCdom{\ell+1}\FreeIsom
			\gamma
\end{align*}
and $\FreeIsom$ preserves the domain operations. %
Likewise for the codomain operations. %

We show by induction that ~ $\FreeIsom$ preserves
\MathText{\ell}{-composition} for ~ $\ell<m$. %
By construction ~ $\FreeIsom$
preserves \MathText{0}{-composition}. %
Assume that ~$\FreeIsom$ preserves 
\MathText{\ell}{-composition}. %
Let ~ $\alpha$ and ~ $\beta$ be ~
\MathText{(\ell+1)}{-composable} \MathText{n}{-cells}. %
Then each has 
\MathText{0}{-decomposition} ~$
	\ZeroComp{y}
	\CompThing{\alpha}{y}
$ and ~$
	\ZeroComp{y}
	\CompThing{\beta}{y}
$ (respectively). %
The following derivation begins by replacing ~
$\alpha$ and ~ $\beta$ with their 
\MathText{0}{-decompositions}, uses 
a basic property of ~ \TextCategories{\omega} 
to proceed from line 1 to line 2, uses
Observation \ref{obs: comp sub} to proceed from 
line 3 to line 4 and uses the induction assumption
to proceed from line 4 to line 5. %
We have 
\begin{align*}
\FreeIsom (\beta\nComp{\ell+1}{}\alpha)
	& =	\FreeIsom 
			(\ZeroComp{y}
			\CompThing{\beta}{\Pred{y},y}
			\nComp{\ell+1}{}
			\ZeroComp{y}
			\CompThing{\alpha}{\Pred{y},y}
			)
\\	& =	\FreeIsom 
			(\ZeroComp{y}
			(\CompThing{\beta}{\Pred{y},y}
			\nComp{\ell+1}{}
			\CompThing{\alpha}{\Pred{y},y}
			))
\\	& =	(
			\FreeIsom
				\SubThing
					{(\CompThing{\beta}{\Pred{y},y}
					\nComp{\ell+1}{}
					\CompThing{\alpha}{\Pred{y},y}
					)}
					{\Pred{y},y}
			)_y
\\	& =	(
			\FreeIsom
				(\SubThing
					{\CompThing{\beta}{\Pred{y},y}}
					{\Pred{y},y}
					\nComp{\ell}{}
				\SubThing{
					\CompThing{\alpha}{\Pred{y},y}}
					{\Pred{y},y}
			))_y
\\	& =	(
			\FreeIsom
				\SubThing{\beta}{\Pred{y},y}
					\nComp{\ell}{}
			\FreeIsom
				\SubThing{\alpha}
					{\Pred{y},y}
			)_y
\\	& =		(
			\FreeIsom
				\SubThing{\beta}{\Pred{y},y})_y
					\nComp{\ell+1}{}
			(\FreeIsom
				\SubThing{\alpha}{\Pred{y},y})_y
\\	& =		\FreeIsom\beta
			\nComp{\ell+1}{}
			\FreeIsom\alpha
\end{align*}
and ~ $L$ preserves composition. %
\end{definition}

	\begin{lemma}
	
	\label{lem: free globcard ograph}
	
The \TextFunctor{\omega} ~
$	L
	\colon\FreeFunctor{X}
	\rightarrow\FreeFunctor
		{\FunctorGlobCardToOGraph X}
$  is an isomorphism. %
\end{lemma}

\begin{proof}

The object set of ~ $\FreeFunctor{X}$ is isomorphic to ~ $X_0$ and the object set of ~
$	\FreeFunctor\FunctorGlobCardToOGraph{X}
$ is ~ $X_0$. %

\textbf{Faithful.} 
We show that ~$\FreeIsom$ is faithful using induction on the ~\MathText{n}{-cells} of 
~$	\FreeFunctor\FunctorGlobCardToOGraph{X}
$. %
Consider morphisms ~$\gamma\colon{Y}\rightarrow{X}
$ and  ~$\gamma'\colon{Y'}\rightarrow{X}
$ of ~ $\FreeFunctor{X}$. %
Suppose that 
~$\FreeIsom\gamma = \FreeIsom\gamma'
$ is a \MathText{0}{-cell} of ~
$	\FreeFunctor\FunctorGlobCardToOGraph{X}
$. %
Then ~$Y$ and ~$Y'$ are singletons, 
~$\gamma$ and ~$\gamma'$ have identical image
and so are isomorphic in ~$\GlobCard/X$. %
Hence they are identical in ~$\FreeFunctor{X}$. 

Assume that ~$\FreeIsom$ is faithful on \MathText{n}{-cells} and let ~
$	\FreeIsom\gamma = \FreeIsom\gamma'
$ be an \MathText{(n+1)}{-cell}. %
We construct an isomorphism ~
$	\alpha\colon{Y}\rightarrow{Y'}
$. %
By Observation ~ \ref{obs: decomp} then ~
$	\gamma
$ and ~
$	\gamma'
$ have canonical 
\MathText{0}{-decompositions} ~
$	\ZeroComp{y} \gamma(y)
$ and 
$	\ZeroComp{y'} \gamma'(y')
$ for ~ $y\in\ExceptFirst{Y_0}$ 
and ~ $y'\in\ExceptFirst{Y'_0}$ respectively. %
By the construction of these compositions and by definition of ~ $\FreeIsom$ then ~ $Y_0$ and ~ $Y'_0$ are isomorphic. %
Such an isomorphism is unique and we have ~
$	\alpha_0\colon Y_0\rightarrow Y'_0
$ as ~$Y_0$ and ~ $Y'_0$ are linear orders. %

Uniqueness of the \MathText{0}{-decompositions} 
in ~$
	\FreeFunctor\FunctorGlobCardToOGraph{X}
$ implies that $
	\FreeIsom \gamma(y)
=	\FreeIsom \gamma'(\delta_0 y)
$ for each ~$
	y\in \ExceptFirst{Y_0}
$. %
By definition of ~ $\FreeIsom$ then 
$	\left( \FreeIsom \gamma(y)(\Pred{y},y) 
	\right)
$ and ~
$	\left( \FreeIsom \gamma'(\delta_0 y)
	(\Pred{\delta_0 y'},\delta_0 y') 
	\right)
$ are identical and by induction the restrictions ~
$	 \gamma(y)(\Pred{y},y) 
$ and ~
$	\gamma'(\delta_0 y)
			(\Pred{\delta_0 y},\delta_0 y) 
$ are identical. %
Then there is a (unique) isomorphism ~ 
$		\SubThing{Y}{\Pred{y},y}
\cong
\SubThing{Y'}{\Pred{\delta_0 y}
					,\delta_0 y}
$ for each ~ $y\in \ExceptFirst{Y_0}$. %
As the corresponding inclusions into ~ $Y\circ u_1$
and ~ $Y'\circ u_1$ (respectively) are jointly epi then we have isomorphisms ~
$	\delta_n\colon Y_n\cong Y'_n
$ for $n\ge 1$. %
The source and target operations are preserved 
and we have ~ $\delta\colon{Y}\cong{Y'}$. %
Hence ~ $\gamma=\gamma'$ and ~ $\FreeIsom$ is 
faithful. %

\textbf{Full.} We show that ~$\FreeIsom$ is full using induction on the ~\MathText{n}{-cells} of 
~$	\FreeFunctor\FunctorGlobCardToOGraph{X}
$. %
Let ~$x$ be a ~\MathText{0}{-cell}. %
Then ~$x$ is an element of ~$X_0$. %
Let ~$Y$ be a globular cardinal with a single element ~$y$ and define a globular morphism ~
$	\gamma\colon{Y}\rightarrow{X}
$ by ~
$	\gamma_0^{\phantom{X}} y=x
$. %
Then ~$L\gamma=x$. 

Assume that ~$\FreeIsom$ is full on ~\MathText{n}{-cells} and let 
~$	x=(x_i)_{i=h+1}^{k}
$ 
be an ~\MathText{(n+1)}{-cell}. %
By induction we have ~\MathText{n}{-cells} ~$\gamma(i)$ such that ~$x_i=\FreeIsom \gamma(i)$.  %
Then as ~$\FreeIsom$ preserves composition we have
~$	\FreeIsom( \circ_{i=h+1}^{k} \gamma(i) ) 
=	\circ_{i=h+1}^{k} x_i
$ which is 
~$	(x_i)_{i=h+1}^{k}
$. %

Hence 
~$	\FreeIsom
	\colon\FreeFunctor{X}
	\rightarrow
		\FreeFunctor\FunctorGlobCardToOGraph{X}
$ is an isomorphism. %
\end{proof}

%% file: theorem222.tex
	\section{Equivalence between $\iPidiNT$ and $\Theta$}

	\label{section: iPidi and theta}

Michael Makkai and Marek Zawadowski demonstrate, in \cite{MZ_DSCD}, a duality between the category ~$\Disks$ as defined by Andr\'{e} Joyal in \cite{AJ_DDTC} and the category, denoted ~$\mathcal{S}$ in \cite{MZ_DSCD}, of \emph{simple} \TextCategories{\omega}. We refer the reader to \cite{MZ_DSCD} for the details, but quote their definition below. Note that ~$[G]$ is the free \TextCategory{\omega} on ~$G$ an ~$\omega\text{-graph}$ which we call a globular set. %

\begin{quote}
Let ~$G$ be an ~$\omega\text{-graph}$. Let us call an element (cell) ~$a$ of ~$[G]$ \emph{maximal} if it is \emph{proper}, that is, not an identity cell, and if the only monomorphisms ~$m \colon H \rightarrow G$ for which ~$a$ belongs to the image of ~$[m]$ are isomorphisms. Intuitively, an element is maximal if it is proper, and the whole graph ~$G$ is needed to generate it. We call ~$G$ \emph{composable} if ~$[G]$ has a \emph{unique} maximal element; in that case, the maximal element may be called the \emph{composite} of the graph.

\ldots

An \TextCategory{\omega} is \emph{simple} if it is of the form ~$[G]$ for a composable ~$\omega\text{-graph}$. The category ~$\mathcal{S}$ is defined as the full subcategory of ~$\wCat$ on the simple \TextCategory{\omega} as objects.
\end{quote}

In Proposition 4.8 of \cite{MZ_DSCD} they demonstrate that an ~$\omega\text{-graph}$ is composable if and only if it is a globular cardinal. %
Hence the objects of ~ $\mathcal{S}$ are \TextCategories{\omega} which are isomorphic to the free \TextCategory{\omega} ~ $\FreeFunctor X$
for some globular cardinal ~$X$. 

	\begin{definition}
	
We define a functor 
\[	\Psi\colon\iPidi\rightarrow\wCat.
\]
Define ~$\Psi$ on objects as the composite~ 
$	\mathfrak{F}\circ\ObjMapiPidiToOGraph
$. %
See Definitions \ref{def: pidi to ograph} and  \ref{def: ordinal free} and Observation \ref{obs: restrict free}, but we make it more explicit here. %
Let ~$H$ be an object of ~$\iPidi$. %
Let~ 
$	\mathcal{G} = \ObjMapiPidiToOGraph{H}
$. %
Then the objects of ~ $\mathcal{G}$ are those of 
~ $H$ 
and the edge-object
$	\mathcal{G}\left(\Pred{i},i\right)
$ 
is ~$
	\ObjMapiPidiToOGraph H(i)
$
for each
$	i\in\ExceptFirst{\Root{\mathcal{G}}}
$. %
The remaining edge-objects are empty ordinal graphs. %
Let 
$	\mathcal{A}=\mathfrak{F}\mathcal{G}
$. %
Then 
$	\mathcal{A}
$ 
is an \TextCategory{\omega} with object set ~
$\Root{\mathcal{G}}$ and hom-objects
\[	\mathcal{A}(i,j) = 
	\prod_{ k =i+1}^{j} 
	\mathfrak{F}\text{ } \ObjMapiPidiToOGraph H(k)
\]
when ~$i < j$ in ~$\Root{\mathcal{G}}$. %
For all ~$i\in\Root{\mathcal{G}}$ then ~$\mathcal{A}(x,x)$ is the terminal \TextCategory{\omega}. %
For ~$i>j$ then ~$\mathcal{A}(i,j)$ is the empty \TextCategory{\omega}. %
Notice that the trivial object 
~$[-1]$ of $\iPidi$ 
is sent to the empty \TextCategory{\omega} which is also initial. %

We define ~$\Psi$ on morphisms by induction on the height of their domain. 
Send morphisms with domain of dimension -1 (minus one) 
to the unique morphism out of the 
empty \TextCategory{\omega}. %
Assume that ~$\Psi$ is defined on morphisms with
domain of height ~$n$ and let
~$	{g}\colon{H}\rightarrow{K}
$ 
have domain of height ~$n+1$. %
We construct an \TextFunctor{\omega} ~$
	\mathcal{F}\colon\Psi{H}\rightarrow\Psi{K}
$ from the data of ~$g$. %
Let ~
$	\mathcal{C}=\Psi{H}
$ 
and ~
$	\mathcal{A}=\Psi{K}
$. %
Then~$\mathcal{C}$ has object set ~$	\Obj H$ and ~$\mathcal{A}$ has object set ~$\Obj K$. %
Define the object map
of ~$\mathcal{F}$ as the object map of ~ $g$. %
We construct an \TextFunctor{\omega} $
	\mathcal{F}_{\Pred{i},i}
	\colon
	\mathcal{C}\left(\Pred{i},i\right) 
	\rightarrow
	\mathcal{A}\left(\mathcal{F} (\Pred{i})
					,\mathcal{F}(i)\right)
$
for each ~ $i\in\ExceptFirst{\Obj\mathcal{C}}$. %
From the definition of ~$\Psi$ on objects we have ~$
	\mathcal{C}\left(\Pred{i},i\right) 
=	\mathfrak{F}\ObjMapiPidiToOGraph H(i)
$ and \[
	\mathcal{A}\left(\mathcal{F}\left(\Pred{i}\right)
				,\mathcal{F}\left(i\right)\right) 
	= \prod_{j=\mathcal{F}\left(\Pred{i}\right) + 1}
		^{\mathcal{F}\left(i\right)}
		\mathfrak{F}\ObjMapiPidiToOGraph K(j)
\]
where ~$
	J_{i}=\{\mathcal{F}\left(\Pred{i}\right) + 1, 
		..., \mathcal{F}\left(i\right)\}
$ is the index set of the product. %
From Observation ~\ref{obs: the pin} we have
~$
J_{i} = \left(g^\wedge\right)^\FiberOver(i)
$
and so 
~$
	g^{\wedge}(j)
=	i
$
for all ~$ j \in J_{i}$. %
We have a morphism ~$
	g(j)
	\colon{H(i)}
	\rightarrow{K(j)}
	\label{eqn: subtree maps}
$ of $\iPidi$
with domain of height ~$n$ for each ~ $j\in J_i$ by definition of ~ $g$. %
By induction, there are \TextFunctors{\omega}
~$	\label{eqn: into the productands}
	\Psi g(j)
	\colon \mathfrak{F}\ObjMapiPidiToOGraph {H(i)}
	\rightarrow \mathfrak{F}\ObjMapiPidiToOGraph {K(j)}
$
which by the universal property of the product give the required morphism
\begin{eqnarray} \label{eqn: into the product}
	\mathcal{F}_{\Pred{i},i}\colon
	\mathfrak{F}\ObjMapiPidiToOGraph H(i)
	\rightarrow
	\prod_{i=\mathcal{F}\left(\Pred{i}\right) + 1}
		^{\mathcal{F}\left(i\right)}
		\mathfrak{F}\ObjMapiPidiToOGraph K(j).
\end{eqnarray} 
We have 
$	\mathcal{F}_{i,j}=
	\prod_{k=i+1}
		^{j}
		\mathcal{F}_{\Pred{k},k}
$ 
by Definition \ref{def: ordinal free} and Observation \ref{obs: restrict free} for ~$i < j$. %
For ~$i$ an object of ~$\mathcal{C}$ then 
~$	\mathcal{F}_{i,i}
$ is the unique morphism into the terminal \TextCategory{\omega}. %
For ~$i>j$ then 
~$	\mathcal{F}_{i,j}
$ has domain the initial (empty) \TextCategory{\omega}. %

\textbf{Preserves composition.} 
Let ~$g'\colon{H}\rightarrow{K}$ and 
~$	g''\colon{K}\rightarrow{L}
$ be composable morphisms of $\iPidi$. %
Put
~$	\mathcal{F}'=\Psi(g')
$, 
~$	\mathcal{F}''=\Psi(g'')
$ and 
~$	\mathcal{F}=\Psi(g)
$ where ~$g=g''\circ g'$. %
Let ~ $i$ be an element of ~ $\Root{H}$, let ~$
	J_{i} 
=	\left({g'}^\wedge\right)^\FiberOver(i)
$ and let ~$
	L_{j} 
=	\left({g''}^\wedge\right)^\FiberOver(j)
$. %
Then $
	L_{i} 
=	\left({g''\circ g'}^\wedge\right)^\FiberOver(i)
$ is identically ~ $\bigcup_{j\in J_i} L_j$. %

We show 
~$	\mathcal{F}=\mathcal{F}''\circ\mathcal{F}'
$ by showing that the upper horizontal composite of
\[\bfig
\qtriangle(0,0)|alr|/->`->`->/<1400,600>%
	[\FreeFunctor\ObjMapiPidiToOGraph H(i)
	`\prod_{j=J_i}
		\FreeFunctor\ObjMapiPidiToOGraph K(j)
	`\FreeFunctor\ObjMapiPidiToOGraph{K}(j)
	;\mathcal{F'}_{\Pred{i},i}
	`\Psi g'(j)
	`\mathrm{pr}
	]
\square(1400,0)|arra|/->`->`->`->/<1400,600>%
	[\prod_{j=J_i}
		\FreeFunctor\ObjMapiPidiToOGraph K(j)
	`\prod_{\ell=L_i}
		\FreeFunctor\ObjMapiPidiToOGraph{L}(\ell)
	`\FreeFunctor\ObjMapiPidiToOGraph{K}(j)
	`\prod_{\ell=L_j}
		\FreeFunctor\ObjMapiPidiToOGraph{L}(\ell)
	;\prod_j 
		\mathcal{F''}_{\Pred{j},j}
	`\mathrm{pr}
	`\mathrm{pr}
	`\mathcal{F''}_{\Pred{j},j}
	]
\qtriangle(1400,-600)|alr|/->`->`->/<1400,600>%
	[\FreeFunctor\ObjMapiPidiToOGraph{K}(j)
	`\prod_{\ell=L_j}
		\FreeFunctor\ObjMapiPidiToOGraph{L}(\ell)
	`\FreeFunctor\ObjMapiPidiToOGraph{L}(\ell)
	;\mathcal{F''}_{\Pred{j},j}
	`\Psi g''(\ell)
	`\mathrm{pr}
	]
\efig\]
is identically ~$\mathcal{F}_{\Pred{i},i}$. %
For each ~$l$ in ~$\Root{L}$ then 
~$	g(\ell)
	\colon H(g^\wedge\ell)
	\rightarrow L(\ell)
$
is 
~$	g''(\ell)\circ g'(j)
$ by definition of composition in ~ $\iPidi$ where ~ $j=g^{''\wedge} \ell$. %
By induction then ~$\Psi g(\ell)$ is the diagonal composite. The construction ending at line \ref{eqn: into the product} gives ~$	
	\mathcal{F}_{\Pred{i},i}$ 
as the unique map 
~$
	\FreeFunctor\ObjMapiPidiToOGraph H(i)
	\rightarrow
	\prod_{\ell=L_i} 
	\FreeFunctor\ObjMapiPidiToOGraph L(\ell)
$ 
which is 
~$	\prod_j 
		\mathcal{F}{''}_{\Pred{j},j} 
		\circ 
		\mathcal{F'}_{\Pred{i},i}
$ as required. %
\end{definition}

	\begin{theorem} \label{thm: 222}

The category ~$\iPidi$ is equivalent to the 
category ~$\Theta_+$ by
\[	\Psi\colon\iPidi\rightarrow\Theta_+. %
\] 
\end{theorem}

	\begin{proof} 
The functor $\Psi$ is 
essentially surjective by 
Lemma \ref{lem: free globcard ograph}. 

	\paragraph{\textbf{Faithful}} 
	
We show ~$\Psi$ is faithful using induction on the dimension of the domain of \TextFunctors{\omega}. %
As the initial object of $\iPidi$ 
is sent to the initial 
\TextCategory{\omega} then ~$\Psi$ is faithful on morphisms with domains of dimension \TextMinusOne\ (minus one) as their domain is the initial object. %
Assume ~$\Psi$ is faithful on morphisms with domain
of dimension ~$n$ and that 
~$	{g},{g'}\colon{H}\rightarrow{K}
$
are parallel morphisms of $\iPidi$ 
with domains of dimension ~$n+1$ such that 
~$	\mathcal{F} = \Psi{g}
$
and
~$	\mathcal{F}' =\Psi{g'}
$
are identical. %
Then 
~$	g=g'
$ as the object maps of ~$\mathcal{F}$ and ~$\mathcal{F}'$ are identical. %

Let ~$\mathcal{C}=\Psi{H}$. %
For each ~
$	{i}\in\ExceptFirst{\Obj \mathcal{C}}$ the construction ending at line \ref{eqn: into the product} gives
\[
	\mathcal{F}_{\Pred{i},i}^{\phantom{'}} 
=	\mathcal{F}_{\Pred{i},i}' 
	\colon
	\mathfrak{F}\ObjMapiPidiToOGraph H(i)
	\rightarrow
	\prod_{j=\mathcal{F}\left(\Pred{i}\right) + 1}
		^{\mathcal{F}\left(i\right)}
		\mathfrak{F}\ObjMapiPidiToOGraph K(j)
\]
where the index     set of the product is 
~$	J_i=\{
		\mathcal{F}\left(\Pred{j}\right) + 1
		, \mathcal{F}\left(\Pred{j}\right) + 2
		, \ldots
		, \mathcal{F}\left(j\right)
		\}
$. %
By construction, composition with the projections out of the product 
gives
\[
	\Psi g(j) =
	\Psi g'(j)
	\colon
	\mathfrak{F}\ObjMapiPidiToOGraph H({i})
	\rightarrow
	\mathfrak{F}\ObjMapiPidiToOGraph K({j})
\]
and by the induction assumption 
~$		g(j)
	=	g'(j)
$
for each ~$ j \in J_{i}$. %
Let 
~$	{J} = \bigcup_{i\in\ExceptFirst
						{\Obj\mathcal{C}}}
				J_{i}
$. %
It remains to show for ~$j\not\in{J}$ that the morphisms 
~$g(j)$ and ~$g'(j)$ determined by ~${j}$ are identical. 

For ~${j}\not\in{J}$ then 
~$	g^\wedge j
$ 
is an endpoint by Observation \ref{obs: outside sent to endpoints}. %
Hence ~$H(g^\wedge j)$ is trivial and
~$	g(j),g'(j)
	\colon{H(g^\wedge j)}
	\rightarrow{K(j)}
$ 
are identical. %
Therefore ~$\Psi$ is faithful. 

	\paragraph{\textbf{Full}}

We show that ~$\Psi$ is full using induction on the dimension of the domain of \TextFunctors{\omega}. %
Let ~$\mathcal{F}\colon\Psi{H}\rightarrow\Psi{K}$ be an \TextFunctor{\omega} where ~$\mathcal{C}=\Psi{H}$. 
We construct a morphism 
~$g$ of $\iPidi$ from the data of ~$\mathcal{F}$ 
such that ~$\Psi g = \mathcal{F}$.
Define the object map of ~$g$ as the object map of ~$\mathcal{F}$. 

Let ~$\mathcal{F}$ be an \TextFunctor{\omega} with 
domain of dimension \TextMinusOne\ (minus one). %
Then ~$\mathcal{F}$ has domain an initial \TextCategory{\omega} and so ~$g$ is the unique morphism
~$	[-1]\rightarrow{K}
$. %
Assume that ~$\Psi$ is full for \TextFunctors{\omega} with domain of dimension ~$n$ and suppose ~
$\mathcal{F}$ has domain of dimension ~$n+1$. 
We have for each object ~
$	j\in\ExceptFirst{\Obj\mathcal{C}}$ morphisms (\TextFunctors{\omega}) 
\[	\mathcal{F}_{\Pred{i},i}
	\colon 
	\mathcal{C}(\Pred{i},i) 
	\rightarrow
	\mathcal{A}(\mathcal{F} (\Pred{i})
				,\mathcal{F}(i))
\]
with domains of dimension ~$n$ which are identically 
\[	\mathcal{F}_{\Pred{i},i}
	\colon 
	\Psi H(i)
	\rightarrow 
	\prod	_{j=\mathcal{F} (\Pred{i})+1}
			^{\mathcal{F}(i)}
	\Psi K(j).
\]
The product has index set 
~$	J_{i}
=		\{
		\mathcal{F}\left(\Pred{j}\right) + 1
		, \mathcal{F}\left(\Pred{j}\right) + 2
		, \ldots
		, \mathcal{F}(i)
		\}
$. %
Composition with the projections of the product gives for each ~$
	j \in J_{i}
$ an \TextFunctor{\omega} \[
	\mathcal{F}(i,j)
	\colon 
	\Psi H(i)
	\rightarrow 
	\Psi K(j). %
\]
By induction there is a morphism 
~$g(i,j)$ such that 
~$	\Psi g(i,j) = \mathcal{F}(i,j)
$ 
for each ~$j\in{J_{i}}$ and each 
~$	i\in\ExceptFirst{\Obj\mathcal{C}}
$. %
Let 
~$	J=\bigcup_{}
				J_{i}
$ where the union is indexed over ~ 
$	\ExceptFirst{\Obj\mathcal{C}}$. %
It remains to define, for each ~${j}\not\in{J}$, morphisms 
~$	g(j)
	\colon H(g^{\wedge} j)
	\rightarrow K(j)
$. %
For such ~${j}$ then 
~$	g^\wedge j
$ 
is an endpoint by Observation ~\ref{obs: outside sent to endpoints}. %
Hence ~$H(g^\wedge j)$ is initial and ~$g(i)$ is determined. Therefore ~$\Psi$ is full. 
\end{proof}

	\begin{corollary} \label{}

The category ~$\iPidiNT$ is equivalent to the 
category ~$\mathcal{S}$. %
\end{corollary}

We have demonstrated that the categories ~$\Disks$ and ~$\Theta$ are dual using the well-known equivalence between intervals and ordinals. %
The category ~ $\iDiscNT$ has been shown equivalent 
to ~$\Disks$ and provides a description of disks as 
inductively defined intervals. %
Similarly ~$\iPidiNT$ has been shown equivalent 
to ~$\Theta$ and provides a description of ~$\Theta$ 
in terms of inductively defined ordinals. %

%% file: general.tex
	\section{Labeled trees}

	\label{section: trees and discs}
	
In this section we define dual categories, 
named ~$\DiscNT$ and ~$\PidiNT$, which are equivalent 
to the category ~  $\Disks$ and the 
category ~ $\Theta$ (respectively). %
To do so we construct categories 
named ~$\Disc$ and ~$\Pidi$, 
which are augmented counterparts 
to ~$\DiscNT$ and ~$\PidiNT$ and show
that they are dual using the equivalence 
between ordinals and intervals. %
They are constructed from trees whose vertices are labeled with intervals and ordinals (respectively) and which satisfy certain conditions. %

We begin by defining the concept of labeled tree 
with appropriate restriction and suspension operations. %
After defining the concept of constrained tree 
we state and prove a general theorem which gives mild 
conditions which allow an equivalence of categories 
to be lifted to an equivalence between certain categories 
of constrained trees labeled by the categories of the
original equivalence. %

We then define dual categories of constrained trees 
labeled by $\mathcal{I}_+$ and $\Delta_+$ which satisfy 
the conditions of the general theorem. %
After defining the concept of cropped tree the 
above equivalence is then restricted to the full 
subcategories with objects the cropped trees of 
positive degree. %
 
	\begin{definition} \label{def: labeled tree} 
	
A \emph{forest ~$(A,F)$ labeled in a category ~$\mathcal{C}$} is a forest equipped with functors ~$F_n\colon A_n \rightarrow \mathcal{C}$ for all ~$n\in\mathbb{N}$ where ~$A_n$ is considered a discrete category. A \emph{tree ~$(A,F)$ labeled in a category ~$\mathcal{C}$} has ~$A$ a tree. 
\end{definition}

	\begin{definition}

	\label{def: labeled tree morphism} 

A \emph{forest morphism 
~$	(f,\alpha)\colon(A,F)\rightarrow (B,G)
$ 
in ~$\mathcal{C}$} is given by a tree map 
~$	f\colon A\rightarrow B
$ 
of trees and a set of natural transformations
\[	\alpha_n\colon F_n\Rightarrow G_n\circ f_n
		\colon A\rightarrow\mathcal{C}
\]
for all ~$n\in\mathbb{N}$. %
A \emph{tree morphism 
~$	(f,\alpha)\colon{A}\rightarrow{B}
$
in ~$\mathcal{C}$} has both $A$ and $B$ labeled trees.

Identity morphisms are given by an identity set map and identity natural transformations. Let 
~$	(f\colon A\rightarrow B, 
	\alpha\colon F\Rightarrow Gf)
$
and 
~$	(g\colon B\rightarrow C, 
	\beta\colon G\Rightarrow Hg)
$ 
be composable morphisms. %
Then ~$(g,\beta)\circ(f,\alpha)$ is defined as 
\[	(g\circ f\colon A\rightarrow C, 
	\beta_f \circ \alpha\colon F\Rightarrow H g f).
\] %

A \emph{forest op-morphism
$	(f,\alpha)
	\colon(B,G)
	\rightarrow(A,F)
$ 
in ~$\mathcal{C}$} is a tree morphism in ~$\mathcal{C}^\op$ and so is a tree map 
~$f\colon A\rightarrow B
$ 
and a set of natural transformations
\[	\alpha_n\colon G_n\circ f_n \Rightarrow F_n
		\colon A\rightarrow\mathcal{C}
\]
for all ~$n\in\mathbb{N}$. %
Let 
~$	(g\colon B\rightarrow C, 
	\beta \colon Hg\Rightarrow G)
$
and
~$	(f\colon A\rightarrow B, 
	\alpha \colon Gf\Rightarrow F)
$ be composable op-morphisms. %
Then ~$(f,\alpha)\circ(g,\beta)$ is defined as 
\[	(g\circ f\colon A\rightarrow C, \alpha\circ\beta_f \colon Hgf \Rightarrow F).
\]

We have the category $\Forest(\mathcal{C})$ of labeled forests and morphisms in $\mathcal{C}$ and its full subcategory $\Tree(\mathcal{C})$ with objects labeled trees in ~$\mathcal{C}$. %
\end{definition}

	\begin{definition} \label{def: rest labeled} 

We define a \emph{restriction} operation on labeled trees. %
The \emph{\SubThingText{(A,F)}{x}} is 
denoted 
~$	(\SubThing{A}{x}
	,\SubThing{F}{x}
	)
$
where ~$(A,F)$ is a labeled forest and ~$x$ is an element of ~$A_n$ has ~$A(x)$ given by 
Definition \ref{def: subtree} and ~$
	\SubThing{F}{x}_{m}
$ defined as the composite 
\[
\SubThing{A}{x}_m
\to/->/<500>^{\mathrm{incl.}}
A_{n+m} 
\to/->/<500>^{F_{n+m}}
\mathcal{C}.
\] %

The 
\emph{\SubThingText{(f,\alpha)}{x}} is denoted 
~$	(\SubThing{f}{x},\SubThing{\alpha}{x})
$
where 
~$	(f,\alpha)\colon{A}\rightarrow{B}
$ is a forest morphism in $\mathcal{C}$ and ~$x$ is an element of ~$A_n$ has ~$\SubThing{f}{x}$ given by Definition \ref{def: subtree} and
~$	\SubThing{\alpha}{x}_m
$ 
given by the composite pasting diagram
\[\bfig
\place(350,-100)[\alpha_{n+m}]
\place(350,-200)[\Rightarrow]
\square(0,0)|alra|/->`->`->`->/
	<700,300>[\SubThing{A}{x}_{m}`\SubThing{B}{f_n x}_{m}`A_{n+m}`B_{n+m}
		;f(x)_{m}
		`\mathrm{incl}
		`\mathrm{incl}
		`f_{n+m}]
\square(0,-300)|alrb|/->`->`->`->/
	<700,300>[A_{n+m}`B_{n+m}
			`\mathcal{C}`\mathcal{C}
		;f_{n+m}`F_{n+m}`G_{n+m}`1]
\efig
\]
where the two vertical composites are ~$\SubThing{F}{i}_{m}$ and ~$\SubThing{G}{f_n x}_{m}$ (respectively). %

The 
\emph{\SubThingText{(f,\alpha)}{x}} is denoted 
~$	(\SubThing{f}{x},\SubThing{\alpha}{x})
$
where ~$(f,\alpha)$ is a forest op-morphism (with 
~$	f\colon A\rightarrow B
$ 
and 
~$	\alpha\colon Gf\Rightarrow F
$) 
and ~$x$ is an element of ~$A_{n}$ has 
~$	\SubThing{f}{x}
$
given by Definition \ref{def: subtree} and has 
~$	\SubThing{\alpha}{x}_m
$ 
given by the composite pasting diagram
\[\bfig
\place(450,-100)[\alpha_{n+m}]
\place(450,-200)[\Rightarrow]
\square(0,0)|alra|/<-`->`->`<-/<900,300>%
	[\SubThing{B}{f_n x}_{m}
	`\SubThing{A}{x}_{m}
	`B_{n+m}
	`A_{n+m}
	;\SubThing{f}{x}_{m}
	`\mathrm{incl}
	`\mathrm{incl}
	`f_{n+m}
	]
\square(0,-300)|alrb|/<-`->`->`->/<900,300>%
	[B_{n+m}
	`A_{n+m}
	`\mathcal{C}
	`\mathcal{C}
	;f_{n+m}
	`G_{n+m}
	`F_{n+m}
	`1
	]
\efig
\]
where the two vertical composites are ~$\SubThing{G}{f_n x}_{m}$ and ~$\SubThing{F}{x}_{m}$ (respectively). %
\end{definition}

	\begin{definition}
	
	\label{def: suspension treeC}

Define a \emph{suspension} functor
\[	\Suspend
	\colon\CatName{Forest}(\mathcal{C})
		\times\mathcal{C}
	\rightarrow\Tree(\mathcal{C})
\]
as follows. %
Given a forest ~$(A,F)$ in $\mathcal{C}$ and an object $c$ in $\mathcal{C}$ then the \emph{\SuspendThingText{A}{c}} denoted
~$	\SuspendThing{A}{c}
$ is ~$(A',F')
$
where ~$A'=\Suspend A$ and ~$F'$ is defined by ~$F'_0(\ast)=c$ and ~$F'_{n+1}=F_{n}$ for $n\in\mathbb{N}$. %

Given a forest morphism 
~$	(f,\alpha)
	\colon(A,F)
	\rightarrow(B,G)
$ and a morphism ~$g\colon{c}\rightarrow{d}$ of ~$\mathcal{C}$ then the 
\emph{\SuspendThingText{f,\alpha}{g}} denoted
~$	\SuspendThing{f}{g} 
$ is ~
$	(f',\alpha')
	\colon \SuspendThing{A}{c}
	\rightarrow \SuspendThing{B}{d}
$
where ~$f'=\Suspend f$ and ~$\alpha'$ is defined 
by ~
$	\alpha'_0(\ast)=g
$ and ~
$	\alpha'_{n+1}=\alpha_{n}$ for $n\in\mathbb{N}
$. %
\end{definition}

	\begin{observation}
	
	\label{obs: id = sub suspend + F}
	
This is the ``labeled'' counterpart to 
Observation \ref{obs: id = sub suspend}. %
The coproduct of a collection of labeled trees 
is a labeled forest and its suspension is a labeled tree. The labeled subtrees of the 
suspension are isomorphic to the trees of the original collection. %
We provide the details below. %

Let $c$ be an object of ~$\mathcal{C}$ and ~$(A(i),F(i))$ be a labeled tree in ~$\mathcal{C}$ with ~$A(i)_{0}=\{x_{i}\}$ for each ~$i$ in a set ${I}$. %
Let
~$	(A',F')
$
be the \SuspendThingText{(\sum A(i),\sum F(i))}{c}. %
By Observation \ref{obs: id = sub suspend} then 
~$	\mathrm{copr}(i)
	\colon A(i)
	\cong \SubThing{A'}{x_i}
$ and the left triangle of
$$\bfig
\Atrianglepair/->`->`->`->`-->/<600,500>%
	[A(i)_n`A'(x_i)_n`\sum A(i)_n`\mathcal{C}
	;\cong`\mathrm{copr}`F(i)_n`\mathrm{incl.}`]
\efig$$
commutes. %
The lower composite is ~$\SubThing{F'}{x_i}_n$ and we have an isomorphism 
\[	(\mathrm{copr(i)}
	, \Id	\colon F(i)
			\Rightarrow \SubThing{F'}{x_i}
				\circ \mathrm{copr}(i)
	)
\]
between ~$(A(i),F(i))$ and 
~$	(\SubThing{A'}{x_i},\SubThing{F'}{x_i})
$. %
\end{observation}

	\begin{definition} 

	\label{def: constrained trees} 
	
A labeled forest ~$(A,p,F)$ in ~$\mathcal{C}$ is said to be \emph{constrained by the functor~$\mathcal{U}\colon\mathcal{C}\rightarrow\CatName{Set}$} when for each ~${n}\in\mathbb{N}$ 
%
%
we have an isomorphism
\[
\lambda_n	
	\colon A_{n+1}
	\cong \text{el}(\mathcal{U}F_n)
\]
where ~$\text{el}(\mathcal{U}F_n)$ is the 
category of elements of ~$\mathcal{U}F_n$ and 
consists of pairs ~$(y,\xi)$ with ~$y\in A_n$
and ~$\xi\in \mathcal{U}F_n(y)$. %
Then ~ $\lambda_n x = (y,\xi)$ where ~
$	y=p_{n}(x)$. %
The set ~$A_{n+1}$ of ~$(n+1)\text{-dimensional}$ vertices of ~$A$ is determined by the labels on its ~$n\text{-dimensional}$ vertices. %
\end{definition}

	\begin{definition}
		
	\label{def: morphisms of constrained trees}
		
A morphism 
~$	(f, \alpha)\colon(A,p,F)\rightarrow(B,q,G)
$ 
between labeled trees in ~$\mathcal{C}$ is said to be \emph{constrained by 
~$	\mathcal{U}
	\colon\mathcal{C}
	\rightarrow\CatName{Set}
$} 
when for each ~$n\in\mathbb{N}$ 
%
%
%
%
%
\[\bfig
\Square(0,0)|alrb|/->`->`->`->/
	[A_{n+1}
	`\mathrm{el}(\mathcal{U}F_n)
	`B_{m+1}
	`\mathrm{el}(\mathcal{U}G_n)
	;\lambda_n
	`f_{n+1}
	`\mathrm{el}(\mathcal{U}\cdot\alpha_n)
	`\eta_n
	]
\efig\]
commutes where the horizontal arrows are the isomorphisms of Definition \ref{def: constrained trees}. %
An \emph{op-morphism ~
$	(f\colon A\rightarrow B
	, \alpha\colon Gf\Rightarrow F)
	\colon(B,q,G)
	\rightarrow(A,p,F)$ between labeled trees in ~$\mathcal{C}$ constrained by ~$\mathcal{U}\colon\mathcal{C}^\op\rightarrow\CatName{Set}$} 
also has that the above diagram commutes. %
As the horizontal maps above are isomorphisms, 
the set map ~$f_{n+1}$ is determined by the data from lower dimensions. %
We have the category $\Restrict(\mathcal{C},\mathcal{U})$ of labeled forests in $\mathcal{C}$ constrained by $\mathcal{U}$. %
\end{definition}

	\begin{definition}
	
The \TextCategory{2} ~ $\TwoCommaCat$ is the 
\TextCategory{2} of small categories over ${Set}$. %
We define a \TextFunctor{2} ~ 
\[	\Restrict\colon\TwoCommaCat\rightarrow\CatName{Cat}
\] as follows. %
Let ~ 
$	\mathcal{U}
	\colon\mathcal{C}
	\rightarrow\CatName{Set}
$ be an object of ~ $\TwoCommaCat$, which we also write as ~ 
$	(\mathcal{C},\mathcal{U})
$, 
and define ~ 
$	\Restrict(\mathcal{C},\mathcal{U})
$ as the category with objects and morphisms given by Definitions \ref{def: constrained trees} and \ref{def: morphisms of constrained trees} (respectively). %
Let ~ $F\colon\mathcal{C}\rightarrow\mathcal{A}$ be a 
morphism of ~ $\TwoCommaCat$ and define
\begin{align*}
	\Restrict F
&	\colon\Restrict(\mathcal{C}
			,\mathcal{U}_\mathcal{C})
	\rightarrow\Restrict(\mathcal{A}
			,\mathcal{U}_\mathcal{A})
\\&	\colon(A,H)\mapsto(A,F\circ H)
\\&	\colon(f,\alpha)\mapsto(f,F\cdot\alpha)
\end{align*}
by post-composition with ~$F$. 
Let ~ $(A,H)$ be a constrained tree. %
Then ~ 
$	A_{n+1}$ is isomorphic to ~ 
$	\mathrm{el}(\mathcal{U}_\mathcal{C}H_n)
$ which is identically ~
$	\mathrm{el}(\mathcal{U}_\mathcal{A}F H_n)
$ by the commutativity required of morphisms in comma categories, in this case ~ 
$	\mathcal{U}_\mathcal{C}
=	\mathcal{U}_\mathcal{A}F 
$. %
Then ~ $(A,FH)$ is a constrained tree. 
Let ~ $(f,\alpha)$ be a constrained morphism. %
Then  ~ 
$	\mathrm{el}(\mathcal{U}_\mathcal{C}
				\cdot \alpha)
$ is identically ~
$	\mathrm{el}(\mathcal{U}_\mathcal{A} F
				\cdot \alpha)
$
and ~ $\Restrict(F)$ is well-defined. %

Let ~ $\gamma\colon F\rightarrow G$ be a \MathText{2}{-cell} (natural transformation) of ~ $\TwoCommaCat$ and define 
\begin{align*}
	\Restrict \gamma
&	\colon\Restrict F
	\Rightarrow\Restrict G
	\colon\Restrict(\mathcal{C}
			,\mathcal{U}_\mathcal{C})
	\rightarrow\Restrict(\mathcal{A}
			,\mathcal{U}_\mathcal{A})
\\&	\colon(A,H)\mapsto(1_A,\gamma_H^{\phantom{X}})
\end{align*}
by post-composition with ~ $\gamma$. %
Naturality of ~ $\Restrict(\alpha)$ follows directly from that of ~ $\alpha$. %

It is easy to see that  ~ $\Restrict$ preserves identities and composition of \MathText{1}{-cells} and of \MathText{2}{-cells}. %
\end{definition}

	\begin{theorem}
	
	\label{thm: general}

Let 
~$	F\colon\mathcal{C}\rightarrow\mathcal{A}
$ 
and 
~$	G\colon\mathcal{A}\rightarrow\mathcal{C}
$ 
be \MathText{1}{-cells} of ~ $\TwoCommaCat$. %
If ~ $F$ and ~$G$ are an adjoint pair then so are ~ $\Restrict F$ and ~ $\Restrict G$. %
Moreover, if ~$F$ and ~$G$ are mutual inverse equivalences (respectively isomorphisms) then ~$\Restrict F$ and ~$\Restrict G$ are mutual inverse equivalences (respectively isomorphisms). %
\end{theorem}

	\begin{proof}

Since ~$\Restrict$ is a 2-functor, it preserves adjunctions, equivalences and isomorphisms.
\end{proof}

Two cases of labeled trees are of interest here: 
labeling by $\Delta_+$ and by $\mathcal{I}_+$. %
In the case of $\Delta_+$ we constrain by the 
functor $
	\mathcal{U} = U\circ (\Slot)^\wedge
$ and for $\mathcal{I}_+$ by $\mathcal{U}=U$ 
where $U$ is the underlying set functor. %
We define an additional requirement on such trees and call the trees satisfying this additional requirement \emph{cropped}.

	\begin{definition}
	
	\label{def: trees and endpoints}
	

Let ~ $(A,p,F)$ be a labeled tree in ~$\Delta_+$, respectively a labeled tree in ~$\mathcal{I}_+$, constrained as stated above. %
An \emph{end element} ~$x\in A_{n+1}$ is one for 
which ~$\lambda_{n,x}$ is either a greatest or least 
element of ~$F_{n}(p_{n}x)^\wedge$ 
(respectively ~$F_{n}(p_{n}x)$). %
We call ~$(A,p,F)$ \emph{cropped} when, for all ~$n\in\mathbb{N}$,
\[	x\text{ is an \Textendelement { }of } A_{n+1}
	\text{ if and only if }
	\mathcal{U} F_{n+1}(x)\text{ is a singleton.} %
\] %
If $(A,p)$ is a tree of degree 0 (zero), so that all the $p_n$ are bijections, then the labeling $F$ of a 
cropped tree $(A,p,F)$ is unique. %
We call these the \emph{trivial} cropped trees. %
\end{definition}

	\begin{definition} \label{def: Disc category}
	 
The category ~$\Disc$ is the full subcategory 
of ${Con}(\mathcal{I}_+,U)$ whose objects are
cropped labeled trees in ~$\mathcal{I}_+$ of 
finite degree. %
The category ~$\DiscNT$ is the full
subcategory of ~$\Disc$ containing the trees of 
positive degree. 
\end{definition}

	\begin{definition} \label{def: Pidi}
	
The category ~$\Pidi$ is the full subcategory 
of ${Con}(\Delta_+,U(\Slot)^\wedge)$ whose objects 
are cropped labeled trees in ~$\Delta^\op_+$ 
of finite degree. %
The category ~$\PidiNT$ is the full
subcategory of ~$\Pidi$ containing the trees of 
positive degree.
\end{definition}

	\begin{observation}
	
	\label{obs: trivial objects}
	
An object ~$(A,F)$ of ~$\Disc$ is trivial if and only 
if ~$F_{0}(\ast)=[0]$. %
An object ~$(A,F)$ of ~$\Pidi$ is trivial if and only 
if ~$F_{0}(\ast)=[-1]$. %
\end{observation}

	\begin{observation}
	
	\label{obs: disc example}
	
We provide an example of an object 
of $\Disc$ of degree 3 (three). %
\[\bfig
\node root(0,0)[\lbrack 2 \rbrack] 

\node L1_0(-800,300)[\lbrack 0 \rbrack] 
\node L1_1( 000,300)[\lbrack 3 \rbrack] 
\node L1_2(+800,300)[\lbrack 0 \rbrack] 

\node L2_0(-800,600)[\lbrack 0 \rbrack] 
\node L2_1(-600,600)[\lbrack 0 \rbrack] 
\node L2_2(-300,600)[\lbrack 1 \rbrack] 
\node L2_3(+250,600)[\lbrack 2 \rbrack] 
\node L2_4(+600,600)[\lbrack 0 \rbrack] 
\node L2_5(+800,600)[\lbrack 0 \rbrack] 

\node L3_0(-800,900)[\lbrack 0 \rbrack] 
\node L3_1(-600,900)[\lbrack 0 \rbrack] 
\node L3_2(-400,900)[\lbrack 0 \rbrack] 
\node L3_3(-200,900)[\lbrack 0 \rbrack] 
\node L3_4(+050,900)[\lbrack 0 \rbrack] 
\node L3_5(+250,900)[\lbrack 1 \rbrack] 
\node L3_6(+450,900)[\lbrack 0 \rbrack] 
\node L3_7(+600,900)[\lbrack 0 \rbrack] 
\node L3_8(+800,900)[\lbrack 0 \rbrack] 

\node L4_0(-800,1200)[\lbrack 0 \rbrack] 
\node L4_1(-600,1200)[\lbrack 0 \rbrack] 
\node L4_2(-400,1200)[\lbrack 0 \rbrack] 
\node L4_3(-200,1200)[\lbrack 0 \rbrack] 
\node L4_4(+050,1200)[\lbrack 0 \rbrack] 
\node L4_5(+190,1200)[\lbrack 0 \rbrack] 
\node L4_6(+310,1200)[\lbrack 0 \rbrack] 
\node L4_7(+450,1200)[\lbrack 0 \rbrack] 
\node L4_8(+600,1200)[\lbrack 0 \rbrack] 
\node L4_9(+800,1200)[\lbrack 0 \rbrack] 

\node L5_0(-800,1500)[\vdots] 
\node L5_1(-600,1500)[\vdots] 
\node L5_2(-400,1500)[\vdots] 
\node L5_3(-200,1500)[\vdots] 
\node L5_4(+050,1500)[\vdots] 
\node L5_5(+190,1500)[\vdots] 
\node L5_6(+310,1500)[\vdots] 
\node L5_7(+450,1500)[\vdots] 
\node L5_8(+600,1500)[\vdots] 
\node L5_9(+800,1500)[\vdots] 

\arrow|l|[L1_0`root;] 
\arrow|l|[L1_1`root;] 
\arrow|l|[L1_2`root;] 

\arrow|l|[L2_0`L1_0;] 
\arrow|l|[L2_1`L1_1;] 
\arrow|l|[L2_2`L1_1;] 
\arrow|l|[L2_3`L1_1;] 
\arrow|l|[L2_4`L1_1;] 
\arrow|l|[L2_5`L1_2;] 

\arrow|l|[L3_0`L2_0;] 
\arrow|l|[L3_1`L2_1;] 
\arrow|l|[L3_2`L2_2;] 
\arrow|l|[L3_3`L2_2;] 
\arrow|l|[L3_4`L2_3;] 
\arrow|l|[L3_5`L2_3;] 
\arrow|l|[L3_6`L2_3;] 
\arrow|l|[L3_7`L2_4;] 
\arrow|l|[L3_8`L2_5;] 

\arrow|l|[L4_0`L3_0;] 
\arrow|l|[L4_1`L3_1;] 
\arrow|l|[L4_2`L3_2;] 
\arrow|l|[L4_3`L3_3;] 
\arrow|l|[L4_4`L3_4;] 
\arrow|l|[L4_5`L3_5;] 
\arrow|l|[L4_6`L3_5;] 
\arrow|l|[L4_7`L3_6;] 
\arrow|l|[L4_8`L3_7;] 
\arrow|l|[L4_9`L3_8;] 

\arrow|l|[L5_0`L4_0;] 
\arrow|l|[L5_1`L4_1;] 
\arrow|l|[L5_2`L4_2;] 
\arrow|l|[L5_3`L4_3;] 
\arrow|l|[L5_4`L4_4;] 
\arrow|l|[L5_5`L4_5;] 
\arrow|l|[L5_6`L4_6;] 
\arrow|l|[L5_7`L4_7;] 
\arrow|l|[L5_8`L4_8;] 
\arrow|l|[L5_9`L4_9;] 
\efig\]

\end{observation}

	\begin{corollary}

	\label{cor: dual disc pidi}
	
The categories $\Disc$ (resp. $\DiscNT$) and $\Pidi$ (resp. $\PidiNT$) are dual.
\end{corollary}

	\begin{proof}
	
The functors ~$(\Slot)^\wedge$ and ~$(\Slot)^\vee$ of the ordinal/interval equivalence and the constraining functors of ~$\Disc$ and ~$\Pidi$ satisfy the hypothesis of Theorem \ref{thm: general}. %
\end{proof}

The categories ~$\Disc$ and ~$\Pidi$ of labeled trees were designed for their similarity with the categories ~${Fam}_\Sigma(\mathcal{I}_+)$ and ~${Fam}_\Pi(\Delta_+)$ (respectively). The objects of ~$\Disc$ and ~$\Pidi$ are trees of intervals, respectively of ordinals, with extra structure and their morphisms are essentially tree morphisms which respect the additional structure. 

	\section{Equivalence of new definitions}

	\label{sec: equiv ind and trees} 

In this last section we demonstrate an equivalence
between ~$\Disc$ and ~$\iDisc$ and, in a 
parallel proof, between ~$\Pidi$ and ~$\iPidi$. 

	\begin{observation}

	\label{obs: reconstruction}
	
A constrained tree is a labeled tree with the requirement (see Definition \ref{def: constrained trees}) that the fiber of a vertex and the underlying set of its label are isomorphic. %
The restriction operation reflects the fibers of all vertices. %
Hence the restriction of a constrained (respectively cropped) tree is constrained (respectively cropped) and the restriction of a  constrained morphism is constrained. %
 
We show that the coproduct of constrained trees 
is constrained. %
Given a collection of constrained trees ~
$(A(i),F(i))$ with comma objects ~$\lambda(i)$ then, by the universal property of coproduct,
$	\sum \lambda_n
	\colon \sum A_{n+1}
	\rightarrow \sum \mathrm{el}(\mathcal{U}F(i)_n) 
$ 
is an isomorphism. %
As the functor ~$\mathrm{el}(\Slot)$ preserves coproducts then  
~$	\sum \mathrm{el}(\mathcal{U}F(i)_n)
$ 
is isomorphic to 
~$	\mathrm{el}(\sum \mathcal{U}F(i)_n) 
$. %
Hence
~$	\sum A(i)_{n+1}
	\cong\sum \mathrm{el}(\mathcal{U}F(i)_n) 
$
and ~$\sum A(i)$ is constrained. %

We show that the coproduct of constrained 
morphisms is constrained. %
Given a collection ~$(f(i),\alpha(i))$ of constrained morphisms then in 
\[\bfig
\square(0,0)|alma|/->`->`->`->/
	<1000,500>%
	[\sum \SubThing{A}{i}_{n+1}
	`\sum \mathrm{el}(\mathcal{U} 
						\SubThing{F}{i}_n
						)
	`\sum \SubThing{B}{j}_{n+1}
	`\sum \mathrm{el}(\mathcal{U} 
						\SubThing{G}{j}_n
						)
	;\cong
	`\sum f(i)_{n+1}
	`\sum \mathrm{el}(\mathcal{U}\cdot \alpha_n)
	`\cong
	]
\square(1000,0)|amrb|/->`->`->`->/
	<1000,500>%
	[\sum \mathrm{el}(\mathcal{U} 
						\SubThing{F}{i}_n
						)
	`\mathrm{el}(\sum \mathcal{U} 
						\SubThing{F}{i}_n
						)
	`\sum \mathrm{el}(\mathcal{U} 
						\SubThing{G}{j}_n
						)
	`\mathrm{el}(\sum \mathcal{U} 
						\SubThing{F}{i}_n
						)
	;\cong
	`\sum \mathrm{el}(\mathcal{U}\cdot \alpha_n)
	`\mathrm{el}(\sum \mathcal{U}\cdot \alpha_n)
	`\cong
	]
\efig
\]
the left square commutes by functoriality of coproduct and the right square commutes by naturality. %
Hence the coproduct of constrained  morphisms  
(and constrained trees) is constrained. %
The coproduct of cropped trees is cropped as the coprojections in ~$\CatName{Set}$ are jointly surjective monomorphisms. %

The suspension ~$\SuspendThing{A}{c}$ is constrained by ~$\mathcal{U}$ if ~
$	A_0\cong\mathcal{U}c
$ and ~$A$ is constrained by ~$\mathcal{U}$. %
The suspension ~$\SuspendThing{(f,\alpha)}{g}$ is constrained by ~$\mathcal{U}$ if ~
$	f_0\cong\mathcal{U}g
$ and ~$(f,\alpha)$ is constrained by ~
$	\mathcal{U}$. %
Similarly, ~$\SuspendThing{A}{c}$ is  cropped if ~$A$ is cropped and ~$A(i)$ is trivial when ~$i$ is an endpoint of ~$\mathcal{U}c$. 
\end{observation}
 	
	\begin{definition}
	
	\label{def: disk to idisc}
	
We define a functor 
\[	\DiscToiDiscFunctor\colon
		\Disc\rightarrow\iDisc
\]
which in Proposition \ref{prop: Disc to iDisc} will be shown to be an equivalence. %

Define ~$\DiscToiDiscFunctor$ on objects 
of ~$\Disc$ using induction on their degree. %
Send each trivial object of ~$\Disc$
to [0] the trivial object of ~$\iDisc$. %
Assume ~$\DiscToiDiscFunctor$ is defined for objects of degree ~$n$ and let ~$(A,F)$ be an object of 
degree ~$n+1$. %
Define an object ~$H$ of ~$\iDisc$ as follows. %
Let $
	\Root{H}=F_{0}(\ast)
$ and let $
	H(\lambda_{0,x})
=	\DiscToiDiscFunctor \SubThing{A}{x}
$ for each ~${x}\in A_1$. %
Define $
	\DiscToiDiscFunctor A$ as ~ $H
$. %
If ~$(f,\alpha)\colon(A,F)\rightarrow(B,G)$ is an isomorphism then ~$F=G$ as ~$\mathcal{I}_+$ is a 
skeletal category. %
By induction then ~$
	\DiscToiDiscFunctor
$ is constant on isomorphism classes. %
Notice that ~$\DiscToiDiscFunctor$ reflects the trivial objects. %

Notice that ~$\DiscToiDiscFunctor A$ is trivial if and only if ~$A$ is trivial. %
Then ~$\lambda_{0,x}$ is an endpoint of ~$U F_{0}(\ast)$ if and only if ~$x$ is an \Textendelement\ of ~$A_{1}$ if and only if ~$F_{1}(x)=[0]$ (by Definition \ref{def: trees and endpoints}) if and only if ~$A(x)$ is trivial (by Observation \ref{obs: trivial objects}) if and only if ~$\DiscToiDiscFunctor A(x)$ is trivial. %
Hence ~$\DiscToiDiscFunctor$ is well-defined on objects;  
~$H(i)$ is trivial if and only if ~$i$ is an endpoint of ~$\Root{H}$. %

Define ~$\DiscToiDiscFunctor$ on morphisms by induction on the degree of their codomains. %
Send each morphism 
$	(f,\alpha)\colon{A}\rightarrow{B}
$ 
of $\Disc$ with codomain of degree 0 (zero) to the 
unique morphism 
$	\DiscToiDiscFunctor{A}\rightarrow[0]
$ of $\iDisc$. %
Assume ~$\DiscToiDiscFunctor$ is defined for morphisms with codomain of degree ~$n$ and let~
$	(f,\alpha)\colon{A}\rightarrow{B}
$ 
have codomain
of degree ~$n+1$. %
We have an interval morphism ~
$	\alpha_{0,\ast}\colon{F_{0}(\ast)}
	\rightarrow{G_{0}(\ast)}
$
and, by induction, have a morphism ~
$	\DiscToiDiscFunctor \SubThing{f}{x}
$ of ~$\iDisc$ for each ~$x$ in ~${A_{1}}$. %
Define a morphism 
$	g
$ 
of ~$\iDisc$ as follows. %
Let 
$	g=\alpha_{0,\ast}
$ 
and let 
$	g(\lambda_{0,x})
	= \DiscToiDiscFunctor \SubThing{f}{x}
$ 
for each 
$	{x}\in A_1
$. %
Define 
$	\DiscToiDiscFunctor(f,\alpha) = g
$.
\end{definition}

	\begin{proposition}

	\label{prop: Disc to iDisc}

The category ~$\Disc$ is equivalent to the category ~$\iDisc$ by
\[	\DiscToiDiscFunctor\colon
		\Disc\rightarrow\iDisc
\]
which is surjective on objects.
\end{proposition}

	\begin{proof}

\textbf{Surjective.} We show 
$	\DiscToiDiscFunctor
$ 
is surjective on objects by induction on the height 
of objects of ~$\iDisc$. %
The trivial objects of ~$\Disc$ map to ~$[0]$ the object 
of ~$\iDisc$ of height \TextZero. %
Assume 
$	\DiscToiDiscFunctor
$
is surjective on objects 
of height ~$n$ and let ~$	H$ have height ~$n+1$. %
By induction there exists an object ~$A(i)$ such that 
~$	\DiscToiDiscFunctor A(i)= H(i)
$
for each ~$i\in\Root{H}$. %
As ~$\DiscToiDiscFunctor$ reflects the trivial object then ~$A(i)$ is trivial if and only if ~$i$ is an endpoint of ~$\Root{H}$. %
By Observation \ref{obs: reconstruction} then 
~$	(A,F)=\Suspend(\sum A(i),\Root{H})
$
is an object of ~$\Disc$. %
We have 
$	F_{0}(\ast)=\Root{H}
$
and 
~$	\SubThing{A}{x_i}\cong{A(i)}
$ 
by Observation \ref{obs: id = sub suspend + F} for each ~$i\in{H}$. %
Then ~$\DiscToiDiscFunctor A = H$ as ~
$\DiscToiDiscFunctor$ is constant 
on isomorphism classes. %
Hence ~$\DiscToiDiscFunctor$ is surjective on objects. %

\textbf{Faithful.} We show ~$\DiscToiDiscFunctor$ is faithful by induction on the height of the
codomain of morphisms of ~$\iDisc$. %
Let 
~$	H=\DiscToiDiscFunctor(A,F)
$,
~$	K=\DiscToiDiscFunctor(B,G)
$ 
and let \[
	(f,\alpha),(f',\alpha')
	\colon(A,F)\rightarrow(B,G)
\] be parallel morphisms of ~$\Disc$. %
Suppose
~$	\DiscToiDiscFunctor(f,\alpha) = 
	\DiscToiDiscFunctor(f',\alpha')
$
is a morphism of ~$\iDisc$ 
with codomain of height \TextZero. %
Then ~$K$ is terminal and 
~$	(f,\alpha)=(f',\alpha')
$ as ~$B$ is also terminal. %
Assume ~$\DiscToiDiscFunctor$ is faithful for morphisms with codomain of height ~$n$ and suppose~
$	\DiscToiDiscFunctor(f,\alpha)
=	\DiscToiDiscFunctor(f',\alpha')
$
has codomain of height ~$n+1$. %
Then 
$	\alpha_{0,\ast} = \alpha'_{0,\ast}
$ 
and so
$	f_{1}=f'_{1}
$
as ~$(f,\alpha)$ and ~$(f',\alpha')$ are constrained morphisms (see Definition~\ref{def: morphisms of constrained trees}). %
By induction we have
~$	\SubThing{f}{x}=\SubThing{f'}{x}
$ and 
~$	\SubThing{\alpha}{x}
=	\SubThing{\alpha'}{x}
$ 
for each~
$	x\in A_1
$. %
Both upper squares commute in the diagram 
$$\bfig
\place(500,-200)[\alpha_{n+1}]
\place(1000,-200)[\alpha'_{n+1}]
\place(500,-300)[\Rightarrow]
\place(1000,-300)[\Rightarrow]
\square(0,0)|alra|/->`->`->`{@{>}@<3pt>}/
	<1500,500>%
	[\SubThing{A}{x}_n
	`\SubThing{B}{f_1 x}_n
	`A_{n+1}
	`B_{n+1}
	;\SubThing{f}{x}_n = \SubThing{f'}{x}_n
	`\mathrm{incl}
	`\mathrm{incl}
	`f
	]
\square(0,0)|alrb|/->`->`->`{@{>}@<-3pt>}/
	<1500,500>%
	[\SubThing{A}{x}_n
	`\SubThing{B}{f_1 x}_n
	`A_{n+1}
	`B_{n+1}
	;\SubThing{f}{x}_n = \SubThing{f'}{x}_n
	`\mathrm{incl}
	`\mathrm{incl}
	`f'
	]
\square(0,-500)|alrb|/`->`->`{@{>}@<-3pt>}/
	<1500,500>%
	[A_{n+1}
	`B_{n+1}
	`\mathcal{C}
	`\mathcal{C}.
	;
	`F_{n+1}
	`G_{n+1}
	`1
	]
\efig$$
As the inclusions are monomorphisms and are jointly surjective then ~$f=f'$. %
The natural transformations ~$\alpha(i)_{n+1}$ and ~$\alpha'(i)_{n+1}$ are identical and are composites of the entire diagram. %
Again, as the inclusions are jointly surjective then ~$\alpha_{n+1}=\alpha'_{n+1}$. %
Hence 
~$\DiscToiDiscFunctor$ is faithful. %

\textbf{Full.} We show ~$\DiscToiDiscFunctor$ is full by induction on the height of the codomain 
of morphisms of ~$\iDisc$. %
The functor ~$\DiscToiDiscFunctor$ is full on morphisms with codomain of height 0 (zero) as
these objects are 
terminal and ~$\DiscToiDiscFunctor$ is surjective on objects. %
Assume ~$\DiscToiDiscFunctor$ is full for morphisms
with codomain of height ~$n$ and let
~$	g
	\colon{\DiscToiDiscFunctor(A,F)}
	\rightarrow{\DiscToiDiscFunctor(B,G)}
$ have codomain of height ~$n+1$. %
Let ~$	H=\DiscToiDiscFunctor(A,F)
$ and ~$	K=\DiscToiDiscFunctor(B,G)
$. %
Then ~$g$ consists of an interval map ~$
	g\colon \Root{H}\rightarrow\Root{K}
$ from ~ $F_0(\ast)$ to ~ $G_0(\ast)$ and for 
each ~$i\in\Root{H}$ a morphism ~$
	g(i)\colon{H(i)}\rightarrow{K(g i)}
$ of ~$\iDisc$. %
Recall that the data of objects ~$(A,F)$ and ~$(B,G)$ 
includes the isomorphisms 
\[
\lambda_n	
	\colon A_{n+1}
	\cong \text{el}({U}F_n)
\qquad\text{and}\qquad
\eta_n	
	\colon B_{n+1}
	\cong \text{el}({U}G_n)
\]
respectively. %
Define ~$f_1$ as the composite 
$$\bfig
\Square|alrb|/-->`->`->`->/%
	[A_1`B_1`F_0(\ast)`G_0(\ast);f_1`\lambda_0`\eta_0`g]
\efig$$
where the vertical maps are isomorphisms. 
By the induction assumption there exists a morphism $
	f(x)
	\colon{A(x)}
	\rightarrow{B(f_1 x)}
$ of ~$\Disc$
with ~$\DiscToiDiscFunctor f(x)=g(i)$ where ~$\lambda_{0,x}^{\phantom{X}}=i$. %
In the following diagram 
$$\bfig
\place(1200,-200)[\sum \alpha(i)]
\place(1200,-300)[\Rightarrow]
\node m1(0,0)[A_{n+1}]
\node m2(700,0)[\sum \SubThing{A}{x}_n]
\node m3(1700,0)[\sum \SubThing{B}{f_1 x}_n]
\node m4(2400,0)[B_{n+1}]
\node u2(700,500)[\SubThing{A}{x}_n]
\node u3(1700,500)[\SubThing{B}{f_1 x}_n]
\node d2(700,-500)[\mathcal{C}]
\node d3(1700,-500)[\mathcal{C}]
\arrow/->/[u2`m1;\mathrm{incl}]
\arrow|m|/->/[u2`m2;\mathrm{copr}]
\arrow/->/[u2`u3;f(x)_n]
\arrow|m|/->/[u3`m3;\mathrm{copr}]
\arrow/->/[u3`m4;\mathrm{incl}]
\arrow/<--/[m1`m2;]
\arrow|l|/->/[m1`d2;F_{n+1}]
\arrow/-->/[m2`m3;]
\arrow|m|/-->/[m2`d2;\sum \SubThing{F}{x}_n]
\arrow/-->/[m3`m4;]
\arrow|m|/-->/[m3`d3;\sum \SubThing{G}{f_1 x}_n]
\arrow|r|/->/[m4`d3;G_{n+1}]
\arrow|b|/->/[d2`d3;1]
\efig$$
where the coproducts are indexed over all ~$x$ in ~$A_1$ then all regions (except the lower square) commute by coproduct. %
The unique morphism ~$\sum A(x)_n\rightarrow A_{n+1}$ is an isomorphism as the inclusions and coprojections are jointly surjective monomorphisms and the inclusions are monomorphisms. %
Define ~$f_{n+1}$ as the middle horizontal composite. %
Both definitions of ~$f_1$ are identical by the uniqueness property of coproduct. %
Then ~$\SubThing{f}{x}$ is (by Definition \ref{def: rest labeled} of a restricted morphism) the upper horizontal morphism and so 
the given morphism ~$f(x)$ is a restricted morphism based on our definition of ~$f$. %
Define ~$\alpha_{n+1}$ as the composite of the lower triangles and square. %
Then ~$\SubThing{\alpha}{x}$ is (by definition) the entire pasting diagram which, as the triangles commute, is identical to the composite of the vertical squares which is by definition ~$\alpha(i)$. %
Hence ~$\DiscToiDiscFunctor f = g$ and ~$\DiscToiDiscFunctor$ is full. %

Therefore we have an equivalence of categories
\[	\DiscToiDiscFunctor\colon
		\Disc\rightarrow\iDisc
\]
which is surjective on objects.
\end{proof}

	\begin{corollary}

The category ~$\DiscNT$ is equivalent to the category ~$\iDiscNT$. %
\end{corollary}

The proofs of propositions \ref{prop: Disc to iDisc} and
\ref{prop: Pidi to iPidi} are nearly identical. %
We include both and list here the two ways in which the categories ~$\iDisc$ and ~$\iPidi$ differ which affect the details of the two proofs. %
First, the trivial object of ~$\iDisc$ 
is the terminal object ~$[0]$ 
and the trivial object of ~$\iPidi$ 
is the initial object ~$[-1]$. %
Second, the subtrees of an object ~$H$ of ~$\iDisc$
are indexed by the elements of ~$\Root{H}$ where the 
subtrees of an object ~$K$ of ~$\iPidi$ are indexed 
by the elements of ~$(\Root{K})^\wedge$. %

	\begin{definition}
	
	\label{def: pidi to ipidi}
	
We define a functor 
\[	\PidiToiPidiFunctor\colon
		\Pidi\rightarrow\iPidi
\]
which in Proposition \ref{prop: Pidi to iPidi} is shown to be an equivalence. %

Define ~$\PidiToiPidiFunctor$ on objects of ~$\Pidi$
using induction on their degree. %
Send each trivial object of ~$\Pidi$ to [-1] 
the trivial object of ~$\iPidi$. %
Assume ~$\PidiToiPidiFunctor$ is defined for objects of 
degree ~$n$ and let ~$(B,G)$ be an object of 
degree ~$n+1$. %
Define an object ~$K$ of ~$\iPidi$ as follows. %
Let 
$	\Root{K}=G_{0}(\ast)
$ 
and let 
$	K(\lambda_{0,x})
=	\PidiToiPidiFunctor \SubThing{B}{x}
$ 
for each ~${x}\in B_1$. %
Define
$	\PidiToiPidiFunctor B$ as ~ $K
$. %
If ~$(f,\alpha)\colon(B,G)\rightarrow(A,F)$ is an isomorphism then ~$F=G$ as ~$\mathcal{I}$ is a 
skeletal category. %
By induction then
~$	\PidiToiPidiFunctor
$
is constant on isomorphism classes. %
Notice that ~$\PidiToiPidiFunctor$ reflects the initial
object. %

Notice that ~$\PidiToiPidiFunctor A$ is trivial if and only if ~$A$ is trivial. %
Then ~$\lambda_{0,x}$ is an endpoint 
of ~$U F_{0}^\wedge(\ast)$
if and only if ~$x$ is an \Textendelement of ~$A_{1}$ 
if and only if ~$F_{1}(x)=[-1]$
(by Definition \ref{def: trees and endpoints}) if and 
only if ~$A(x)$ is trivial (by 
Observation \ref{obs: trivial objects}) if and only if ~$
	\PidiToiPidiFunctor A(x)
$ is trivial. %
Hence ~$\PidiToiPidiFunctor$ is well-defined on 
objects; ~$H(i)$ is trivial if and only if ~$i$ is an 
endpoint of ~$(\Root{H})^\wedge$. %

Define ~$\PidiToiPidiFunctor$ on morphisms by induction on the degree of their domains.
Send each morphism ~$
	(f,\alpha)\colon{B}\rightarrow{A}
$ with domain of degree 0 (zero) to the unique morphism
~$	[-1]\rightarrow\PidiToiPidiFunctor{A}
$ of ~$\iPidi$. %
Assume ~$\PidiToiPidiFunctor$ is defined for morphisms with domain of degree ~$n$ and let ~$
	(f,\alpha)\colon(B,G)\rightarrow(A,F)
$ have domain of degree ~$n+1$. %
We have an ordinal morphism ~$
	\alpha_{0,\ast}\colon{G_{0}(\ast)}
	\rightarrow{F_{0}(\ast)}
$ and, by induction, have a morphism ~$
	\PidiToiPidiFunctor \SubThing{f}{x}
$ of ~$\iPidi$ for each ~$x$ in ~${A_{1}}$. %
Define a morphism $g$ of ~$\iPidi$ as follows. %
Let ~$
	g=\alpha_{0,\ast}
$ and let ~$
	g(\lambda_{0,x})
	= \PidiToiPidiFunctor \SubThing{f}{x}
$ for each ${x}\in A_1$. %
Define $
	\PidiToiPidiFunctor(f,\alpha) = g
$. %
\end{definition}

	\begin{proposition}
	
	\label{prop: Pidi to iPidi}

The category ~$\Pidi$ is equivalent to the category ~$\iPidi$ by
\[	\PidiToiPidiFunctor\colon
		\Pidi\rightarrow\iPidi
\]
which is surjective on objects.
\end{proposition}

	\begin{proof}

\textbf{Surjective.} We show 
$	\PidiToiPidiFunctor
$ 
is surjective on objects by induction on the height 
of objects of ~$\iPidi$. 
The trivial objects of ~$\Pidi$ map to ~$[-1]$ 
the object of ~$\iPidi$ of height \TextZero. %
Assume 
$	\PidiToiPidiFunctor
$
is surjective on objects of ~$\iPidi$ 
of height ~$n$ and let ~$K$ have height ~$n+1$. %
By induction there exists a disk ~$B(i)$ such that 
~$	\PidiToiPidiFunctor B(i)= K(i)
$
for each ~$i\in(\Root{K})^\wedge$.
As ~$\PidiToiPidiFunctor$ reflects the trivial object then ~$B(i)$ is trivial if and only if ~$i$ is an endpoint of ~$(\Root{K})^\wedge$.
By Observation \ref{obs: reconstruction} then 
\[	(B,G)=\Suspend(\sum B(i),(\Root{K})^\wedge)
\]
is an object of ~$\Pidi$. 
We have 
$	G_{0}(\ast)=\Root{K}
$
and 
~$	\SubThing{B}{x_i}\cong{B(i)}
$ 
by Observation \ref{obs: id = sub suspend + F} for each ~$i\in{K}$.
Then ~$\PidiToiPidiFunctor B = K$ as ~$\PidiToiPidiFunctor$ is constant on isomorphism classes. %
Hence ~$\PidiToiPidiFunctor$ is surjective on objects. %

\textbf{Faithful.} We show ~$\PidiToiPidiFunctor$ is faithful by induction on the height of the 
domain of  morphisms of ~$\iPidi$. %
Let 
~$	K=\PidiToiPidiFunctor(B,G)
$ 
and
~$	H=\PidiToiPidiFunctor(A,F)
$
and let \[
	(f,\alpha),(f',\alpha')
	\colon(B,G)\rightarrow(A,F)
\] be parallel morphisms of ~$\Pidi$. %
Suppose ~$
	\PidiToiPidiFunctor(f,\alpha) = 
	\PidiToiPidiFunctor(f',\alpha')
$ is a morphism of ~$\iPidi$
with domain of height \TextZero. %
Then ~$K$ is initial
and ~$	(f,\alpha)=(f',\alpha')
$ as ~$B$ is also initial.
Assume ~$\PidiToiPidiFunctor$ is faithful for morphisms 
with domain 
of height ~$n$ and suppose~
$	\PidiToiPidiFunctor(f,\alpha)
=	\PidiToiPidiFunctor(f',\alpha')
$
has domain 
of height ~$n+1$. %
Then 
$	\alpha_{0,\ast} = \alpha'_{0,\ast}
$ 
and so
$	f_{1}=f'_{1}
$
as ~$(f,\alpha)$ and ~$(f',\alpha')$ are constrained morphisms (see Definition~\ref{def: morphisms of constrained trees}). %
By induction we have
~$	\SubThing{f}{x}=\SubThing{f'}{x}
$ and 
~$	\SubThing{\alpha}{x}
=	\SubThing{\alpha'}{x}
$ 
for each~
$	x\in A_1
$. %
Both upper squares commute in the diagram 
$$\bfig
\place(500,-200)[\alpha_{n+1}]
\place(1000,-200)[\alpha'_{n+1}]
\place(500,-300)[\Rightarrow]
\place(1000,-300)[\Rightarrow]
\square(0,0)|alra|/<-`->`->`{@{<-}@<3pt>}/
	<1500,500>%
	[\SubThing{B}{f_1 x}_m
	`\SubThing{A}{x}_m
	`B_{m+1}
	`A_{m+1}
	;\SubThing{f}{x}_m = \SubThing{f'}{x}_m
	`\mathrm{incl}
	`\mathrm{incl}
	`f
	]
\square(0,0)|alrb|/<-`->`->`{@{<-}@<-3pt>}/
	<1500,500>%
	[\SubThing{B}{f_1 x}_m
	`\SubThing{A}{x}_m
	`B_{m+1}
	`A_{m+1}
	;\SubThing{f}{x}_m = \SubThing{f'}{x}_m
	`\mathrm{incl}
	`\mathrm{incl}
	`f'
	]
\square(0,-500)|alrb|/`->`->`{@{>}@<-3pt>}/
	<1500,500>%
	[B_{m+1}
	`A_{m+1}
	`\mathcal{C}
	`\mathcal{C}.
	;
	`G_{n+1}
	`F_{n+1}
	`1
	]
\efig$$
As the inclusions are monomorphisms and are jointly surjective then ~$f=f'$. %
The natural transformations ~$\alpha(i)_{n+1}$ and ~$\alpha'(i)_{n+1}$ are identical and are composites of the entire diagram. %
Again, as the inclusions are jointly surjective then ~$\alpha_{n+1}=\alpha'_{n+1}$. %
Hence 
~$\PidiToiPidiFunctor$ is faithful. %

\textbf{Full.} We show ~$\PidiToiPidiFunctor$ is full by induction on the height of
the domain of morphisms of ~$\iPidi$. %
The functor ~$\PidiToiPidiFunctor$ is full on morphisms
with domain of height 0 (zero) as 
these objects are
initial
and ~$\PidiToiPidiFunctor$ is surjective on objects. %
Assume ~$\PidiToiPidiFunctor$ is full for morphisms 
with domain 
of height ~$n$ and let
$	g\colon{\DiscToiDiscFunctor(B,G)}
	\rightarrow{\DiscToiDiscFunctor(A,F)}
$ 
have domain of
height ~$n+1$. %
Let
~$	H=\DiscToiDiscFunctor(A,F)
$ and
~$	K=\DiscToiDiscFunctor(B,G)
$. 
Then ~$g$ consists of an ordinal map ~
$	g
	\colon\Root{K}
	\rightarrow\Root{H}
$ 
from ~ $G_0(\ast)$ to ~ $F_0(\ast)$ and for each ~$j\in(\Root{H})^\wedge$
a morphism  
$	g(j)
	\colon K(g^\wedge j)
	\rightarrow H(j)
$ of ~$\iPidi$. %
Recall that the data of objects ~$(A,F)$ and ~$(B,G)$ 
includes the isomorphisms 
\[
\lambda_n	
	\colon A_{n+1}
	\cong \text{el}(U(\Slot)^\wedge F_n)
\qquad\text{and}\qquad
\eta_n	
	\colon B_{n+1}
	\cong \text{el}(U(\Slot)^\wedge G_n)
\]
respectively. %
Define ~$f_1$ as the composite 
$$\bfig
\Square|alrb|/<--`->`->`<-/%
	[B_1`A_1`G_0(\ast)^\wedge`F_0(\ast)^\wedge
	;f_1`\lambda_0`\eta_0`g^\wedge]
\efig$$
where the vertical maps are isomorphisms. 
By the induction assumption there exists a morphism
$	f(x)
	\colon{B(f_1 x)}
	\rightarrow{A(x)}
$ 
with ~$\PidiToiPidiFunctor f(x)=g(j)$
where ~$\eta_{0,x}^{\phantom{X}}=j$.
In the following diagram 
$$\bfig
\place(1200,-200)[\sum \alpha(i)]
\place(1200,-300)[\Rightarrow]
\node m1(0,0)[B_{n+1}]
\node m2(700,0)[\sum \SubThing{B}{f_1 x}_n]
\node m3(1700,0)[\sum \SubThing{A}{x}_n]
\node m4(2400,0)[A_{n+1}]
\node u2(700,500)[\SubThing{B}{f_1 x}_n]
\node u3(1700,500)[\SubThing{A}{x}_n]
\node d2(700,-500)[\mathcal{C}]
\node d3(1700,-500)[\mathcal{C}]
\arrow/->/[u2`m1;\mathrm{incl}]
\arrow|m|/->/[u2`m2;\mathrm{copr}]
\arrow/<-/[u2`u3;f(x)_n]
\arrow|m|/->/[u3`m3;\mathrm{copr}]
\arrow/->/[u3`m4;\mathrm{incl}]
\arrow/<--/[m1`m2;]
\arrow|l|/->/[m1`d2;G_{n+1}]
\arrow/<--/[m2`m3;]
\arrow|m|/-->/[m2`d2;\sum \SubThing{G}{x}_n]
\arrow/-->/[m3`m4;]
\arrow|m|/-->/[m3`d3;\sum \SubThing{F}{f_1 x}_n]
\arrow|r|/->/[m4`d3;F_{n+1}]
\arrow|b|/->/[d2`d3;1]
\efig$$
where the coproducts are indexed over all ~$x$ in ~$A_1$
then all regions (except the lower square) commute by coproduct. %
The unique morphism 
~$	\sum A(x)_n
	\rightarrow A_{n+1}
$
is an isomorphism as the inclusions and coprojections are jointly surjective monomorphisms and the inclusions are monomorphisms. %
Define ~$f_{n+1}$ as the middle horizontal composite. %
Both definitions of ~$f_1$ are identical by the uniqueness property of coproduct. %
Then ~$\SubThing{f}{x}$ is (by Definition \ref{def: rest labeled} of a restricted morphism) the upper horizontal morphism and so 
the given morphism ~$f(x)$ is a restricted morphism based on our definition of ~$f$. %
Define ~$\alpha_{n+1}$ as the composite of the lower triangles and square. %
Then ~$\SubThing{\alpha}{x}$ is (by definition) the entire pasting diagram which, as the triangles commute, is identical to the composite of the vertical squares which is by definition ~$\alpha(i)$. %
Hence ~$\PidiToiPidiFunctor f = g$ and ~$\PidiToiPidiFunctor$ is full. %

Therefore we have an equivalence of categories
\[	\PidiToiPidiFunctor\colon
		\Pidi\rightarrow\iPidi
\]
which is surjective on objects.
\end{proof}

	\begin{corollary}

The category ~$\PidiNT$ is equivalent to the category ~$\iPidiNT$. %
\end{corollary}

The category ~$\DiscNT$ provides a description of disks 
as trees with vertices labeled by intervals. %
Similarly ~$\PidiNT$ provides a description of ~$\Theta$ 
in terms of trees with vertices labeled by ordinals. %


%% file: bibliography.tex

%% file: disc.bbl
\clearpage

\begin{thebibliography}{9} 

\bibitem{MB_MGC}
Michael Batanin, 
Monoidal globular categories as a natural environment for the theory of weak $n$-categories, \emph{Advances in Mathematics}
\textbf{136} (1998) 39--103
 
\bibitem{MB_UPMT} 
Michael Batanin, Ross Street,
The universal property of the multitude of trees,
\emph{Journal of Pure and Applied Algebra}
\textbf{154} (2000) 3--13

\bibitem{CB_CNHC}
Clemens Berger,
A cellular nerve for higher categories,
\emph{Advances in Mathematics}  
\textbf{169} (2002) 118--175
 
\bibitem{AJ_DDTC}
Andr\'{e} Joyal,
\emph{Disks, duality and $\Theta\text{-categories}$},
Preprint (1997)

\bibitem{SM_CWM}
Saunders Mac Lane, 
\emph{Categories for the Working Mathematician},
2nd edition,
Springer-Verlag (1998)

\bibitem{MZ_DSCD}
Mihaly Makkai and Marek Zawadowski,
Duality for simple $\omega$-categories and disks,
\emph{Theory and Applications of Categories}
\textbf{8}  (2001) 114--243
 
\bibitem{RS_AOS} 
Ross Street, 
The algebra of oriented simplexes,
\emph{Journal of Pure and Applied Algebra}
\textbf{49} (1987) 283--335

\bibitem{RS_PTGS} 
Ross Street, 
The petit topos of globular sets, 
\emph{Journal of Pure and Applied Algebra}
\textbf{154} (2000) 299--315

\bibitem{DV_CS}
Dominic Verity,
\emph{Complicial sets: characterising the simplicial nerves of strict $\omega$-categories},
Memoirs of the American Mathematical Society  \textbf{193}  (2008) 

\end{thebibliography}
